\documentclass[11pt]{amsart}
\usepackage{amsfonts,latexsym,amsthm,amssymb,amsmath,amscd,euscript,tikz, tikz-cd}
\usepackage[alphabetic, msc-links, bibtex-style, nobysame]{amsrefs}
\usepackage{stackengine}
\usepackage{framed}
\usepackage{xfrac}
\usepackage[makeroom]{cancel}
\usepackage{faktor}
\usepackage{braket}
\usepackage{pgf,tikz,pgfplots}
\pgfplotsset{compat=1.17}
\usepgfplotslibrary{fillbetween} 
\usepackage{mathrsfs}
\usetikzlibrary{arrows}
\definecolor{wrwrwr}{rgb}{0.3803921568627451,0.3803921568627451,0.3803921568627451}
\definecolor{rvwvcq}{rgb}{0.08235294117647059,0.396078431372549,0.7529411764705882}
\definecolor{mblue}{rgb}{0.2, 0.3, 0.8}
\definecolor{morange}{rgb}{1, 0.5, 0}
\definecolor{mgreen}{rgb}{0.1, 0.4, 0.2}
\definecolor{mred}{rgb}{0.5, 0, 0}
\definecolor{ForestGreen}{RGB}{34,139,34}
\usepackage{hyperref}
\hypersetup{colorlinks=true,citecolor=ForestGreen,linkcolor = blue,urlcolor =black,linkbordercolor={1 0 0}}
\usepackage{float}
\usepackage{mathtools}
\mathtoolsset{showonlyrefs}

\numberwithin{equation}{section}

\usepackage{lmodern}
\usepackage{supertabular}
\usepackage{verbatim}
\usepackage{amssymb}
\usepackage{enumerate}
\usepackage{stmaryrd}
\usepackage{bbm}
\usepackage{mathtools}
\usepackage[nopatch=footnote]{microtype}
\usepackage{setspace}
\usepackage{tikz,tikz-cd}
\usetikzlibrary{matrix,calc,positioning,arrows,decorations.pathreplacing,patterns,knots}

\newcommand{\la}{\langle}
\newcommand{\rg}{\rangle}

\newtheorem{theorem}{{Theorem}}[section]
\newtheorem*{theorem*}{Theorem}
\newtheorem{lemma}[theorem]{Lemma}
\newtheorem{proposition}[theorem]{Proposition}

\newtheorem{corollary}[theorem]{Corollary}
\newtheorem*{corollary*}{Corollary}
\newtheorem*{conjecture}{Conjecture}
\newtheorem*{question}{Question}

\theoremstyle{definition}
\newtheorem{definition}{Definition}
\newtheorem{remark}{Remark}

\usepackage{lmodern,url,enumerate,mathtools,microtype}
\usepackage[hmargin = 1in,vmargin=1in]{geometry}
\usepackage{graphicx}
\usepackage{subcaption}

\newcommand{\ve}{\varepsilon}

\newcommand{\mr}[1]{{\rm #1}}
\newcommand{\cF}{\mathcal{F}}
\newcommand{\cH}{\mathcal{H}}
\newcommand{\cJ}{\mathcal{J}}
\newcommand{\cL}{\mathcal{L}}

\newcommand{\bC}{\mathbb{C}}
\newcommand{\bF}{\mathbb{F}}
\newcommand{\bG}{\mathbb{G}}\newcommand{\bH}{\mathbb{H}}

\newcommand{\bO}{\mathbb{O}}\newcommand{\bP}{\mathbb{P}}
\newcommand{\bR}{\mathbb{R}}
\newcommand{\RR}{\mathbb{R}}
\newcommand{\bS}{\mathbb{S}}\newcommand{\bT}{\mathbb{T}}

\newcommand{\bZ}{\mathbb{Z}}

\newcommand{\nc}{\newcommand}

\nc{\on}{\operatorname}
\nc{\p}{\partial}
\nc{\ol}{\overline}
\nc{\ul}{\underline}
\nc{\pa}{\partial}

\nc{\pb}{\partial_b}
\nc{\pc}{\partial_c}
\nc{\pd}{\partial_d}
\nc{\pe}{\partial_e}
\nc{\pf}{\partial_f}
\nc{\pg}{\partial_g}
\nc{\ph}{\partial_h}
\nc{\pari}{\partial_i}
\nc{\pj}{\partial_j}
\nc{\pk}{\partial_k}
\nc{\pl}{\partial_l}
\nc{\pell}{\partial_\ell}
\nc{\parm}{\partial_m}
\nc{\pn}{\partial_n}
\nc{\po}{\partial_o}
\nc{\pp}{\partial_p}
\nc{\pq}{\partial_q}
\nc{\pr}{\partial_r}
\nc{\ps}{\partial_s}
\nc{\pt}{\partial_t}
\nc{\pu}{\partial_u}
\nc{\pv}{\partial_v}
\nc{\pw}{\partial_w}
\nc{\px}{\partial_x}
\nc{\py}{\partial_y}
\nc{\pz}{\partial_z}

\nc{\Spec}{\on{Spec}}

\nc{\sn}{\mr{sn}}
\nc{\cn}{\mr{cn}}
\nc{\dn}{\mr{dn}}

\nc{\thru}{~\!\!--~\!\!}

\numberwithin{equation}{section}
\makeatletter
\@addtoreset{equation}{section}
\makeatother

\title[Singularity models of the Bernoulli problem]{
Singularity models for the Bernoulli free boundary problem from isoparametric hypersurfaces}
\date{\today}

\author[Benjy Firester]{Benjy Firester}
\address{{\href{mailto:benjyfir@mit.edu}{benjyfir@mit.edu}} \hfill Department of Mathematics, MIT}
\author[Raphael Tsiamis]{Raphael Tsiamis}
\address{
{\href{mailto:r.tsiamis@columbia.edu}{r.tsiamis@columbia.edu}}\hfill Department of Mathematics, Columbia University}
\author[Zihui Zhao]{Zihui Zhao}
\address{{\href{mailto:zhaozh@jhu.edu}{zhaozh@jhu.edu}}\hfill Department of Mathematics, Johns Hopkins University}
\thanks{The third-named author is partially supported by the NSF grant DMS-2350351.}

\begin{document}

\begin{abstract}
    We develop a general construction of homogeneous solutions to the Bernoulli free boundary problem, as well as general extremal domains on the sphere, from isoparametric foliations of the sphere. Our construction produces rich families of infinitely many new examples with sophisticated topologies connected to minimal surfaces of the sphere by smooth families of interpolating capillary surfaces, which include novel singularity models in low dimensions and recover most known homogeneous one-phase solutions. We also introduce a geometric representation for every such homogeneous map and establish a geometric rigidity theorem: for every prescribed isoparametric foliation, and in particular for cohomogeneity-one subgroups of $O(n)$, its radial geometry uniquely determines the corresponding homogeneous solution. Finally, we study the density of the constructed solutions and conjecture that the lowest density among non-flat cones in a given dimension $n \geq 5$ is attained by an $O(k) \times O(n-k)$-invariant cone.
\end{abstract}

\maketitle

\section{Introduction} 
We develop a general framework to construct intricate singularity models for the Bernoulli free boundary problem. 
The Bernoulli free boundary problem studies the critical points, stable points, and minimizers of the Alt-Caffarelli energy
\[ 
\cJ(u, \Omega_0) := \int_{\Omega_0} |\nabla u|^2 + \chi_{\{u>0\}} \, dx, 
\]
subject to a prescribed Dirichlet boundary condition. 
A function $u$ that is a critical point of the Alt-Caffarelli energy solves the overdetermined elliptic problem
\begin{equation}\label{eqn:bernoulli-problem}
\left\{\begin{array}{ll}
	\Delta u = 0, & \text{ in } \{ u>0\} \cap \Omega_0, \\
	\;|\nabla u| = 1, & \text{ on } \partial\{u>0\} \cap \Omega_0.
\end{array} \right. 
\end{equation} 
The work of Alt-Caffarelli and Weiss~\cites{AC,Weiss} reduce the study of the boundary regularity (for minimizers), i.e.~the regularity of the \emph{free boundary} $\partial\{u>0\}$, via a blow-up argument, to the existence of non-flat (minimizing) homogeneous solutions of degree one, namely
\begin{equation}\label{eq:hom}
	\left\{\begin{array}{ll}
	\;\Delta U = 0, & \text{ in } \{ U>0\} \subset \RR^n, \\
	\;\;|\nabla U| = 1, & \text{ on } \partial\{U>0\}, \\
	U(\rho \omega) = \rho \, U(\omega), & \text{ where } \rho = |x|, \; \omega = \frac{x}{|x|}.
	\end{array} \right. 
\end{equation} 
For this reason, homogeneous solutions to \eqref{eq:hom}, or equivalently, their restriction $\phi := U|_{\bS^{n-1}}$ to the sphere satisfying
\begin{equation}\label{eq:homsph}
	\left\{\begin{array}{ll}
	-\Delta_{\mathbb S^{n-1}} \phi = (n-1) \phi, & \text{ in } \{\phi>0\} \\
	|\nabla \phi|= 1, & \text{ on } \partial \{\phi>0\}
\end{array}
 \right. 
\end{equation}
are exactly the models of singularities of solutions at the free boundary.

The axisymmetric \textit{De Silva-Jerison cone} was proved to be energy minimizing in dimension $n \geq 7$ and unstable in dimension $n < 7$, see \cites{desilva-jerison-cones, CJK}. 
Recently,~\cites{FTW-1,FTW-stability-one-phase} constructed new families of homogeneous solutions and proved their \textit{strict} minimality, including a second minimizing example in the critical dimension $7$. 
On the other hand, Caffarelli-Jerison-Kenig and Jerison-Savin~\cites{CJK, jerison-Savin} demonstrated that all stable cones in dimensions $n=3$ and $4$, respectively, must be flat, thus showing that the critical dimension for singularity formation is at least $5$.
Understanding this gap and the optimal regularity is a major open question in the study of the one-phase free boundary problem. 
Recent works show that one-phase cones appear as one endpoint of a geometric interpolation through capillary surfaces with minimal surfaces as the other endpoint, where the optimal regularity in dimension $7$ is a classical result, see ~\cites{FTW-1,new-minimal-surfaces}, as well as~\cite{capillary-FBP}, which constructs the minimal capillary cone analogues generated by isoparametric foliations.
We also highlight other notable connections between the one-phase free boundary problem and the minimal surface theory, including the construction of solutions and the uniqueness and rigidity of singularity models~\cites{entire-hairpins, traizet , jerison-kamburov, n-dim-catenoid, one-phase-simon-solomon}.

Besides the De Silva-Jerison cone, other homogeneous solutions and their stability have been studied in \cites{hong-singular, FTW-stability-one-phase}, as well as other interesting examples in low dimensions in \cite{cubicFBP} (in $\RR^5$) and \cite{hines-kolesar-mcgrath} (in $\RR^3$ and $\RR^4$).
In this paper, we develop a general construction of new homogeneous solutions in $\RR^n$, coming from isoparametric foliations of the sphere $\bS^{n-1}$.
Our technique recovers all previously known examples, with the exception of~\cite{hines-kolesar-mcgrath}, whose free boundary on the sphere $\bS^2$ has many connected components. 
Notably, our construction produces novel examples in $\bR^5$ and $\bR^6$.
In a follow-up paper \cite{six-way}, we will study the stability, minimality, and linear spectrum of these solutions.
In particular, we prove that these solutions are indeed unstable in low dimensions, supporting the expectation that $n=7$ is the critical dimension for regularity.

From another perspective, the equation \eqref{eq:homsph} is also an example of a shape optimization problem. 
Given a complete Riemannian manifold $(N,g)$ and some $0 < m < \textup{Vol}(N,g)$, the \textit{shape optimization problem} seeks minimizers or critical points of the first Dirichlet eigenvalue $N \supset \Omega \mapsto \lambda_1(\Omega)$, among connected domains with prescribed volume $\operatorname{vol}(\Omega) = m$. 
When $N = \mathbb{S}^{n-1}$, the well-known Faber-Krahn inequality shows that the minimizers are geodesic balls.
More generally, a domain $\Omega\subset \mathbb{S}^{n-1}$ with smooth boundary is a critical point with respect to volume-preserving deformations if and only if its first Dirichlet eigenfunction $\phi$ solves the overdetermined problem
\begin{equation}\label{eq:extremal}
	\left\{\begin{array}{ll}
	-\Delta_{\mathbb S^{n-1}} \phi = \lambda_1(\Omega) \phi, & \text{ in } \Omega \\
    \phi= 0, & \text{ on } \partial\Omega \\
	\frac{\partial \phi}{\partial\nu}= -1, & \text{ on } \partial \Omega,
\end{array}\right. 
\end{equation}
for $\nu$ the outward-pointing unit normal.
The sets $\Omega$ supporting such solutions are called \emph{extremal domains} in the literature. 
Thus, singular models of the Bernoulli free boundary problem correspond to extremal domains in the sphere, whose first Dirichlet eigenvalue is equal to $(n-1)$.
For any $\lambda>0$, our construction also gives a rich family of extremal domains with first Dirichlet eigenvalue $\lambda$.

We present a construction of novel solutions of the above problems, notably~\eqref{eq:hom} and ~\eqref{eq:extremal}, with complicated topology, which utilizes the geometry of isoparametric foliations.
An isoparametric foliation of the sphere consists of parallel hypersurfaces with constant mean curvature, together with two \textit{focal submanifolds} of higher codimension.
Every isoparametric foliation is equipped with parameters $(g,m_1,m_2)$, where $g \in \{ 1,2,3,4, 6\}$ denotes the number of distinct principal curvatures and $M_1, M_2$ are the focal submanifolds of codimensions $m_1+1$ and $m_2+1$, respectively, such that $m_1 + m_2 = \frac{2(n-2)}{g}$.
In Section~\ref{subsec:isoparametric}, we discuss important properties of such foliations.

For any given isoparametric parameters $(g, m_1, m_2)$ with $g \geq 2$, we consider the hypergeometric function $f_{M}$ with a zero at $t_{M} > 0$, and define
\allowdisplaybreaks{
\begin{equation}\label{eqn:hypergeometric-isoparametric}
\begin{split}
    f_{M_1}(t) &:= {}_2 F_1 \Bigl( \frac{n-1}{g} , - \frac{1}{g} ; \frac{m_1 + 1}{2} ; t^2 \Bigr) = {}_2F_1 \Bigl( \frac{m_1 + m_2}{2} + \frac{1}{g} , - \frac{1}{g} ; \frac{m_1 + 1}{2} ; t^2 \Bigr), \\
    c_{M_1} &:= \frac{2}{g} \frac{1}{\sqrt{1 - t_{M_1}^2} \, |f'_{M_1}(t_{M_1})|}.
\end{split}
\end{equation}}
Let $f_{M_2}, c_{M_2}$ denote the pair $(f_{M_1}, c_{M_1})$ after exchanging $m_1$ and $m_2$, namely
\begin{equation}\label{eqn:hypergeometric-isoparametricconj}
    f_{M_2}(t) := {}_2F_1 \Bigl( \frac{m_1 + m_2}{2} + \frac{1}{g} , - \frac{1}{g} ; \frac{m_2+1}{2} ; t^2 \Bigr), \qquad c_{M_2} = \frac{2}{g} \frac{1}{\sqrt{1 - t_{M_2}^2} |f'_{M_2} (t_{M_2})|} .
\end{equation}
where $t_{M_2}$ is the unique positive root of $f_{M_2}(t)$.

Geometrically, the transposition $m_1 \leftrightsquigarrow m_2$ preserves the global isoparametric foliation associated to $M$, while exchanging the two focal submanifolds $M_1$ and $M_2$.
In our situation, only one of the focal submanifolds $M_i$ will be contained in the positivity set of $f_{M_1}$ and $f_{M_2}$, respectively.

\begin{theorem}\label{thm:isoparametric-foliation-one-phase-cone}
Let $M \subset \bS^{n-1}$ be an isoparametric hypersurface with parameters $(g, m_1, m_2)$, producing a foliation of $\bS^{n-1}$ with normal foliation parameter $s \in [0, \frac{\pi}{g}]$.
For $g \geq 2$, there exist exactly three homogeneous solutions of the one-phase free boundary problem in $\bR^n$ whose restriction to the unit sphere is constant along leaves of the isoparametric foliation generated by $M$:
\begin{align}
    U_{M_1} (x) &:= c_{M_1} \, \rho  f_{M_1} \Bigl( \sin \frac{g s(\omega)}{2}  \Bigr), \qquad  U_{M_2} (x) := c_{M_2} \, \rho  f_{M_2} \Bigl( \cos \frac{g s(\omega)}{2}  \Bigr)  \tag{OP.I}\label{eqn:U-m-foliation} \\
    \bar{U}_M (x) &:= \bar{c}_M \, \rho \bar{f}_M \Bigl( \sin \frac{g s(\omega)}{2}  \Bigr) , \tag{OP.II}\label{eqn:U-bar-m-second-solution}
\end{align}
where for each $x = \bR^n$ we denote $\rho := |x|$ and $\omega := \frac{x}{|x|}$.
The functions $f_{M_1}(t), f_{M_2}(t)$ are given by~\eqref{eqn:hypergeometric-isoparametric} and \eqref{eqn:hypergeometric-isoparametricconj}, while $\bar{f}_M(t)$ has two zeroes at $\bar{t}_1 < \bar{t}_2 \in (0,1)$ and $\bar{c}_M = \frac{2}{g} \frac{1}{\sqrt{1 - \bar{t}_i^2} \, |\bar{f}'_M ( \bar{t}_i) |}$.

For the Type I solution $U_{M_i}$, the free boundary is the cone over a regular leaf $M_s$; for the Type II solution $\bar{U}_M$, it is the cone over the disjoint union of two regular leaves $M_{s_1} \sqcup M_{s_2}$. 
\end{theorem}
When $g=2$, \eqref{eqn:U-m-foliation} recover the solutions studied in \cite{FTW-stability-one-phase}, which are proved to be unstable when $n\leq 6$ and strictly stable when $n\geq 7$. 
Their (strict) minimality is further studied in \cite{FTW-1}. In \cite{hong-singular}, Hong studied homogeneous solutions that correspond to \eqref{eqn:U-bar-m-second-solution} in the case $g=2$ and $m_1=m_2$.
When $g=3$ and $n=5$, \eqref{eqn:U-m-foliation} recovers the solution constructed in \cite{cubicFBP}. 
Our constructions give completely new solutions in every dimension $n\geq 5$.
We note that if $m_1=m_2$, the two Type I profiles $f_{M_1}, f_{M_2}$ in \eqref{eqn:U-m-foliation} coincide after a change of variables $t \leftrightsquigarrow \sqrt{1-t^2}$. However, the two focal manifolds of the foliation may not be isometric even when $m_1=m_2$, so they can produce different homogeneous solutions; see Section \ref{section:preliminaries}.

More generally, we also construct families of extremal domains on the sphere.

\begin{theorem}\label{thm:extremal}
Consider an isoparametric foliation $\{M_s\}_{s\in [0,\frac{\pi}{g}]}$ of $\mathbb{S}^{n-1}$ with parameters $(g,m_1, m_2)$, and given $\lambda>0$. There exist exactly three solutions $\phi$ to \eqref{eq:extremal} that are constant along leaves of the foliation, and their corresponding extremal domains $\Omega :=\{\phi>0\}$ are
    \[ \bigcup_{0\leq s< \tau_1}M_s, \qquad \bigcup_{\sigma_2< s\leq \frac{\pi}{g}}M_s, \qquad \bigcup_{\bar\sigma< s<\bar\tau} M_s \, . \]
    Moreover, the endpoints $\tau_1, \sigma_2 , \bar\sigma, \bar\tau \in (0,\frac{\pi}{g})$ are uniquely determined by $(g,m_1,m_2)$ and $\lambda$.
\end{theorem}

\subsection{Density and complexity of homogeneous solutions}
We recall the \textit{Weiss monotonicity formula}, which says that for a given variational solution $u$ of the problem~\eqref{eqn:bernoulli-problem}, a free boundary point $x\in \partial\{u>0\}\cap \Omega_0$, and any radius $r$ such that $B(x,r) \subset \Omega_0$, the quantity
\begin{equation}\label{eqn:weiss-quantity}
    W(u;x, r) := \frac{1}{r^n}\int_{B_r(x)} (|\nabla u|^2 + \chi_{\{u > 0\}}) - \frac{1}{r^{n+1}}\int_{\partial B_r(x)} u^2 \, d \cH^{n-1}
\end{equation}
is monotone non-decreasing, with equality $W(u;x, r_1) = W(u;x, r_2)$ for some $r_1 < r_2$ if and only if $u$ is one-homogeneous with vertex at $x$, cf.~\cites{Weiss, KW-variation}.
Moreover, the blow-up analysis shows that for any sequence $r_k \to 0$, modulo passing to a subsequence $\frac{u(x+r_k z)}{r_k} $ converges to a one-homogeneous solution $U$ of 
\eqref{eq:hom}, called a \emph{tangent function} of $u$ at $x$. 
In general, the homogeneous solution $U$ depends on the sequence of radii $r_k$ and may not be unique; however for any tangent function $U$, the area ratio of its positive set is uniquely determined as
\begin{equation}\label{eqn:define-density}
\Theta(U) := \frac{\mathcal{H}^{n-1}(\{U>0\} \cap \mathbb{S}^{n-1})}{\mathcal{H}^{n-1}(\mathbb{S}^{n-1})} = \frac{|\{U>0\}\cap B_1 |}{\omega_n} = \frac{1}{\omega_n} \lim_{r \downarrow 0} W(u;x,r),
\end{equation}
where $|\cdot |$ denotes the Lebesgue measure in $\RR^n$. 
We call $\Theta(U)$ the \textit{density} of the solution $U$, which quantifies the complexity and the energy level of the free boundary point $x$.

We recall that homogeneous solutions of the Bernoulli problem are equivalent to solutions of the shape optimization problem~\eqref{eq:homsph} for a spherical domain $\{ U > 0 \} \cap \bS^{n-1}$ with first Dirichlet eigenvalue $(n-1)$.
Sperner's isoperimetric inequality on $\bS^{n-1}$ implies that such domains satisfy
\[ \mathcal{H}^{n-1}(\{U>0\} \cap \mathbb{S}^{n-1}) \geq \mathcal{H}^{n-1}( \mathbb{S}^{n-1}_+), \]
with equality if and only if $U= \la a, x \rg_+$ for some unit vector $a \in \mathbb{R}^n$, and hence $\Theta (U) \geq \frac{1}{2}$, with equality if and only if $U$ is the flat solution.
In fact, the work of De Silva~\cite{desilva-allard}*{Theorem 1.1} implies the existence of a dimensional constant $\epsilon_n > 0$ such that any non-flat one-phase cone in $\bR^n$ has density $\Theta(U) \geq \frac{1}{2} + \epsilon_n$; see also~\cite{gafa-velichkov}*{Lemma 5.3}.
Therefore, it is natural to ask:
\begin{question}
Among non-flat homogeneous solutions in $\mathbb{R}^n$, which one has the smallest density?
\end{question}
This problem is analogous to the Yau-Solomon conjecture in minimal surface theory, which asserts that the smallest density ratio among non-flat minimal hypercones is attained by the Simons cone $C(\bS^k \times \bS^k)$ in $\bR^{2(k+1)}$ and by the cone $C(\bS^k \times \bS^{k+1})$ in $\bR^{2k+3}$, cf.~\cite{Yau-openproblem}*{Problem 31}. 
This conjecture has been confirmed asymptotically as the ambient dimension tends to $\infty$, thanks to the important works~\cites{Ilmanen-White-density, Zhu-density, Bernstein-Wang-density}, which apply ideas from mean curvature flow.

The density conjecture is also closely related to the shape optimization problem in the sphere, which seeks the minimal or critical values of the first Dirichlet eigenvalue $\lambda_1(\Omega)$, among domains $\Omega$ with fixed $\mathcal{H}^{n-1}$-area. 
In our situation, the first Dirichlet eigenvalue $\lambda_1(\Omega)$ is prescribed as $(n-1)$, due to the homogeneity of the free boundary problem, and we are interested in the smallest area that can be achieved by non-trivial, i.e., non-geodesic, extremal domains.

It has long been believed among experts on free boundary problems that the axisymmetric cone plays an analogous role for the Bernoulli problem to the one played by the Simons cone for minimal surfaces, likely due to its discovery as the first minimizing example in~\cite{desilva-jerison-cones}.
However, the work of the first two authors with Yipeng Wang~\cites{FTW-1 , FTW-stability-one-phase } indicates a different picture: the Type I cones with bi-orthogonal symmetry, corresponding in the notation of Theorem~\ref{thm:unified-theorem} to $U_{M_i}$ with $(M_1, M_2) \cong (\bS^{m_2}, \bS^{m_1})$ the focal submanifolds of the foliation of $\bS^{n-1}$ by Clifford tori $\bS^{m_1}(\sin s) \times \bS^{m_2}(\cos s)$, are also strictly stable and strictly minimizing in dimension $n \geq 7$.
In particular, the cone $U_{\bS^1 \subset \bS^{n-1}}$ with $(g,m_1,m_2) = (2, n-3,1)$, denoted by $U_{n,n-2}$ in~\cite{FTW-stability-one-phase}, maximizes the first stability eigenvalue among this family.
Driving this paradigm shift further, numerical evidence suggests that the axisymmetric solution does not minimize the density among cones in dimensions $n \geq 5$. Moreover, asymptotically as $n \to \infty$, we prove the following in Section \ref{section:density-of-cones}.
 
\begin{proposition}\label{prop:intro-density}
    Let $\bar{\Theta}_{1,n-2,n-2}$ denote the density of the axisymmetric De Silva-Jerison cone, and $\Theta_{2,m_1,m_2}$ denote the density of the Type I solution constructed in Theorem \ref{thm:isoparametric-foliation-one-phase-cone} with parameters $(g=2, m_1, m_2)$. They have the asymptotics
    \[ \lim_{m_1, m_2 \to \infty} \Theta_{2,m_1,m_2} \simeq 0.77785 < 0.80876 \simeq \lim_{n \to \infty} \bar{\Theta}_{1,n-2,n-2}. \]
\end{proposition}

\subsection{Geometric properties of homogeneous solutions}
In addition to our new constructions and rigidity results, we introduce a framework for studying homogeneous solutions of the one-phase problem with isolated singularities through the geometry of complete hypersurfaces in $\bR^n$.

\begin{definition}\label{def:LinkOfCone}
Given a homogeneous solution $U$ of the problem~\eqref{eq:hom} whose free boundary has an isolated singularity, we define its \textit{radial graph} $\Sigma$ to be the hypersurface 
\[
\Sigma := \{x : U(x) = 1\} \subset \bR^n
\]
with induced metric $g$.
We define $\Gamma_U := \mr{Isom}(\Sigma)$ to be the isometry group of the radial graph, and $\Gamma_U^0$ to be the connected component of $\Gamma_U$ containing the identity.
\end{definition}
Geometrically, $(\Sigma, g)$ is a complete, non-compact, properly embedded, real-analytic hypersurface~\cite{kinderlehrer-nirenberg-spruck}.
By homogeneity, the original solution $U$ can be reconstructed from its radial graph.
Indeed, let us write $x = \rho \omega$ for $(\rho, \omega) \in (0,\infty) \times \bS^{n-1}$ and $\Omega := \{ U > 0\} \subset \bR^n$ as well as $\Omega^S := \Omega \cap \bS^{n-1}$ for the positive phase of $U$.
The spherical Faber-Krahn inequality implies that $\Omega^S$ is connected, cf.~\cite{jerison-kamburov}*{\S 5}.
We may therefore write
\[
\Sigma = \left\{ \frac{\omega}{U(\omega)} : \omega \in \Omega^S \right\} = \left\{ \frac{x}{U(x)} : x \in \Omega \right\}
\]
and $\Sigma$ is a connected, star-shaped radial graph over $\Omega^S$, whose rescalings $\{ \lambda \Sigma\}_{\lambda > 0} = \{ x : U(x) = \lambda \}$ foliate the conical region $\Omega = \{ U > 0 \}$.
The geometry of $\Sigma$ gives a complete characterization of the homogeneous solution $U$ and encapsulates its complexity and symmetries.
The isometry group $\Gamma_U := \textup{Isom}(\Sigma)$ and the component $\Gamma^0_U$ provide an algebraic invariant for homogeneous solutions: notably, $\Gamma_U = \textup{E}(n-1)$ is the Euclidean group if and only if $U(x) = \la a ,  x \rg_+$ a flat half-space solution, see Lemma \ref{lemma:euclidean-structure}.
The full regularity theory of the Bernoulli free boundary problem can be interpreted as a classification and rigidity problem for the radial graphs, with the stability condition transforming into a nonlocal operator on $\Sigma$.

The radial graph $\Sigma$ is foliated by the level sets of the distance functions to the origin.
In our constructions, which arise from isoparametric foliations, this foliation consists of dilates of the leaves $M_s$, possibly including the focal submanifold $M_i$.
This occurs for the Type I solutions, for which the locus of points achieving the minimum distance to the origin is a dilate of one $M_i$.
In the Type II case, $\Sigma_U$ has two ends, and each fiber $\Sigma \cap \bS^{n-1}(R)$ is the union of two regular leaves $M_s$.
The region closest to the origin $\frac{1}{\mr{dist}(\Sigma,0)}\{x \in \Sigma : |x| = \mr{dist}(\Sigma, 0)\}$ represents the distinguished leaf where the profile curve $\bar{f}_M$ is maximized.

In Section~\ref{sec:RadialGraph}, we provide a more detailed description of the geometry of the radial graphs of the solutions $U_{M_i}, \bar{U}_M$ constructed in Theorem~\ref{thm:isoparametric-foliation-one-phase-cone}.
Furthermore, combining Theorem~\ref{thm:unified-theorem} with the above discussion allows us to characterize singularity models purely by the algebraic invariant $\Gamma^0_U$.

\begin{theorem}\label{thm:unified-theorem}
Consider a homogeneous solution $U$ of the Bernoulli free boundary problem with radial graph $\Sigma_U := \{ U=1 \}$.
\begin{enumerate}[(i)]
\item Consider an isoparametric foliation $\cF$ with $g \geq 2$, if the radial projection of $\Sigma_U$ and its radial function are invariant along the leaves of $\cF$.
Then, either $\Sigma$ has one end and $U$ is one of the solutions $\{ U_{M_1}, U_{M_2} \}$, or $\Sigma$ has two ends and $U$ agrees with $\bar{U}_{M}$ after an ambient isometry.
In the first case, the focal submanifold is determined by
\[
\tfrac{1}{\textup{dist}(\Sigma,0)} \{ x \in \Sigma : |x| = \textup{dist}(\Sigma,0) \} = M_i \subset \bS^{n-1}.
\]
\item If the foliation $\cF$ consists of orbits of a compact connected group $G \subset O(n)$, then the classification $(i)$ holds for radial graphs $\Sigma_U$ that are $G$-invariant.
\end{enumerate}
In particular, if $\textup{SO}(n-1) \subseteq \textup{Isom}(\Sigma_U)$, then $U$ is either a flat half-space solution or the axisymmetric De Silva-Jerison cone.
\end{theorem}

It is interesting to compare the classification Theorem~\ref{thm:unified-theorem} with the corresponding result in classical minimal surface theory.
In $\bR^3$, a fundamental result due to Riemann and Enneper~\cites{riemann, enneper} states that the catenoid and the Riemann minimal surface are the only surfaces foliated by circles.
Shiffman~\cite{shiffman} extended this classification to \textit{compact} minimal surfaces with boundary bounded between two parallel planes.
In higher dimensions, the situation is different: Jagy~\cite{jagy} proved that the only minimal surface in $\bR^{n+1}$ with $n \geq 3$ foliated by spheres is the catenoid, and the higher-dimensional Riemann minimal examples of Kaabachi-Pacard~\cite{kaabachi-pacard} are not foliated by $\bS^{n-1}$.
In contrast, Theorem~\ref{thm:unified-theorem} provides an analogue of the minimal surface uniqueness results for the free boundary problem that is valid in every dimension.

The study of radial graphs associated with homogeneous solutions of the Bernoulli free boundary problem is also inspired by the ambitious program proposed by Jerison~\cite{jerison}.
Jerison's approach seeks to obtain more refined information about the regularity of stable critical points of the Alt-Caffarelli functional and the classification of singularity models by relating them to the level sets of harmonic functions.
The technique of radial graphs, together with Theorem~\ref{thm:unified-theorem}, provides a complementary geometric approach to the classification of homogeneous singularity models. 
Rather than constructing a desingularization of a given cone by global minimizers as the work of De Silva-Jerison-Shahgholian~\cite{hardt-simon-DSJ}, we recover the cone from the geometry and symmetry of a distinguished positive level set. 
This approach applies without an assumption of stability or minimality, and gives a complete classification within every isoparametric symmetry class.

\section{Preliminaries}\label{section:preliminaries}

\subsection{Isoparametric hypersurfaces of the sphere}\label{subsec:isoparametric}

We consider \textit{isoparametric foliations} of the sphere $\bS^{n-1}$ which are constant mean curvature (CMC) foliations given as the level sets of $F: \bS^{n-1} \to \bR$ which are \textit{isoparametric}, meaning they satisfy
\begin{equation}\label{eqn:distance-properties}
\bigl|\nabla^{\bS^{n-1}} F \bigr|^2 = a(F), \qquad \Delta_{\mathbb S^{n-1}} F = b(F)
\end{equation}
for smooth functions $a, b$.
Any CMC foliation by parallel hypersurfaces is locally isoparametric.

Isoparametric foliations of $\bS^{n-1}$ have been characterized, and largely classified, in a series of landmark works.
The primary classification is due to Cartan and M\"unzner~\cites{Cartan1939, cheese1, cheese2}, which associates to each foliation canonical parameters $(g, m_1, m_2)$.
Here, $g \in \{1,2,3,4,6\}$ represents the number of distinct principal curvatures while $(m_1, m_2)$ are the multiplicities of the first two principal curvatures, where $m_i = m_{i+2}$ modulo $g$ for further ones.
These data satisfy the relation
\begin{equation}\label{eqn:counting-multiplicities}
    g( m_1 + m_2) = 2 (n-2).
\end{equation}
The foliation degenerates at two lower dimensional distinguished manifolds $M_1$ and $M_2$ called the \textit{focal submanifolds}, of codimensions $m_1 + 1$ and $m_2+ 1$ respectively.
Every regular leaf $M_s$ is connected and a parallel surface of any other, so there is a natural parameter $s \in (0,\frac{\pi}{g})$ which measures the distance to $M_1$ and sweeps out every regular leaf $M_s$.
One distinguished leaf $M_s$ will be minimal, so we denote it $M$.
Expressing $\bR^n$ as the homogenization of $\bS^{n-1}$, we can extend any such foliation of the sphere to the $\bR^{n}$ by dilations, so every point is associated with some leaf $M_s$ by rescaling.
We extend the notion of isoparametric functions on $\bS^{n-1}$ to $\bR^n$ by dilations, which are called \textit{$\cF$-invariant}, see \cite{wang-on-a-class}*{Definition 2.2}.

We briefly record known classifications and the topology of these isoparametric foliations, which will help classify and identify the solutions in Theorem~\ref{thm:isoparametric-foliation-one-phase-cone}.
\begin{enumerate}[$\bullet$]
\item $g=1$: The foliation comes from the inductive structure of spherical sine-suspensions $\bS^n = \Sigma \bS^{n-1}$.
The focal surfaces are the north and south poles.
\item $g=2$: Generalizing the previous example, this foliation comes from the bi-orthogonal symmetry $\textup{O}(n-k) \times \textup{O}(k) \subset \textup{O}(n)$ whose natural invariant hypertori $\bS^{n-k-1}(\sin s)\times \bS^{k-1}(\cos s)$ form the Clifford family, and the focal manifolds are the high codimension unit spheres in the invariant orthogonal subspaces $\bR^{n-k}$ and $\bR^{k}$. 
\item $g=3$: The Cartan foliations arise from the four real division algebras $\bF \in \{ \RR, \bC, \bH, \bO\}$ with $m_1 = m_2 = m := \dim_{\bR} \bF \in \{1,2,4, 8\}$.
There is an additional symmetry from the antipodal involution that isometrically identifies the two focal submanifolds, which are the two Veronese embeddings of the projective plane $\bF \bP^2 \hookrightarrow \bS^{3m+1}$.
The regular leaves $M_s$ are the flag manifolds
\[
\textup{SO}(3) / (\bZ_2 \oplus \bZ_2) , \qquad \textup{SU}(3) / \bT^2, \qquad \textup{Sp}(3) / \textup{Sp}(1)^3, \qquad F_4 / \textup{Spin}(8).
\]
\item $g=4$: The isoparametric hypersurfaces are either of OT-FKM type~\cites{ferus-karcher-munzner, OTg4} or belong to one of the two exceptional homogeneous families with $(m_1,m_2)\in\{(2,2),(4,5)\}$. 
In the OT-FKM case, $(m_1, m_2) = (m , k \, \delta(m) - m - 1)$, where $\delta(m)$ is the dimension of a irreducible $C_{m-1}$-module over the Clifford algebra with $(m-1)$ generators, computed by Bott periodicity as
\begin{equation}\label{eqn:values-using-bott-delta}
\delta ( 8k +j) = 2^{4k} \delta(j), \qquad \text{where} \quad \delta(1) = 1, \; \delta(2) = 2, \; \delta(3) = \delta(4) = 4, \; \delta(j) = 8, \; 5 \leq j \leq 8.
\end{equation}
Moreover, the focal submanifold $M_1$ is the intersection of quadric hypersurfaces while $M_2$ is a sphere bundle over $\bS^m$.
The regular leaves satisfy $M \cong M_1 \times \bS^{m_1}$.

In the two exceptional homogeneous cases,
\begin{align*}
(m_1,m_2)&=(2,2):\qquad
M_1\cong \bC\bP^3, && M_2\cong \widetilde{\bG}(2,5)\cong Q^3, \\
(m_1,m_2)&=(4,5):\qquad
M_1\cong \mathrm{U}(5)/(\mathrm{Sp}(2)\times \mathrm{U}(1)),&&
M_2\cong \mathrm{U}(5)/(\mathrm{SU}(2)\times \mathrm{U}(3)).
\end{align*}
The regular leaves of these foliations are given by
\[
M_{(4,2,2)} \cong \textup{SO}(5) / \bT^2, \qquad M_{(4,4,5)} \cong \textup{U}(5)/ ( \textup{SU}(2) \times \textup{SU}(2) \times \textup{U}(1))
\]
where $\bT^2$ is the maximal torus in $\textup{SO}(5)$.

\item $g=6$: The multiplicities are equal and $m_1=m_2\in\{1,2\}$.
When $m=1$, the unique foliation of $\bS^7$ is homogeneous and arises as the inverse image of the Cartan isoparametric foliation with $(g, m_1,m_2) = (3,1,1)$ under the Hopf fibration $\bS^7 \to \bS^4$.
The two focal manifolds are diffeomorphic to $\bR\bP^2\times \bS^3$, while the regular leaves are the homogeneous Hopf lifts of $\textup{SO}(3)/ \bZ_2^{\oplus 2}$.
For $m=2$, the only known example is homogeneous in $\bS^{13}$ with regular leaves $G_2/(\mathrm{U}(1)\times \mathrm{U}(1))$ and focal manifolds $G_2/\mathrm{U}(2)^+$ and $G_2/\mathrm{U}(2)^-$, where $G_2/\mathrm{U}(2)^-\cong Q^5\cong \widetilde{\bG}(2,7)$ is diffeomorphic to the complex quadric and the oriented Grassmannian of $2$-planes in $\bR^7$.
The full classification when $(g,m) = (6,2)$ was studied in the works~\cites{miyaoka , siffert-AGAG}, with~\cite{tang-6-2} recently presenting a short alternative argument to complete the isoparametric classification. 
\end{enumerate} 

\subsection{ \texorpdfstring{$\cF$}{F}-invariant solutions}\label{sec:minimal-surface-equation}

We now study the problems~\eqref{eq:hom} and~\eqref{eq:extremal} for functions that are invariant with respect to an isoparametric foliation $\cF$.
Let us parametrize $\cF$-invariant homogeneous solution as $U(\rho, \omega) = \rho \phi( s (\omega))$, where $\rho = |x|$, $\omega = \frac{x}{|x|}$ and $s \in [0, \frac{\pi}{g}]$ is the parameter along the isoparametric foliation such that $\omega \in M_s$.
Given an isoparametric foliation $\{ M_s \}$ of $\bS^{n-1}$ with principal curvature parameters $(g, m_1, m_2)$, we can use the normal parametrization $s$ for each leaf to compute the induced volume density
\begin{equation}\label{eqn:isoparametric-volume-density}
v(s) = C \left( \sin \tfrac{gs}{2} \right)^{m_1} \left( \cos \tfrac{gs}{2} \right)^{m_2}.
\end{equation}
The mean curvature of the leaf $M_{s}$ can therefore be computed as
\begin{equation}\label{eqn:h-sigma}
H(s) = - \frac{d}{d\sigma}\bigg\vert_{\sigma = s} \log v(\sigma) = \frac{g}{2} \left(m_2 \tan \frac{g s}{2} - m_1 \cot \frac{g s}{2} \right).
\end{equation}
The spherical Laplacian for any $\cF$-invariant function $\phi(s)$ becomes an ODE
\begin{equation}\label{eqn:laplacian-of-f}
\Delta_{\mathbb S^{n-1}} \phi = \phi'' + ( \partial_s \log \sqrt{\det g_s} ) \phi' =  \phi'' - H \phi' = \phi'' + \tfrac{v'}{v} \phi'
\end{equation}
where $g_s$ is the induced metric on the leaf $M_s$.

By~\eqref{eqn:h-sigma} and~\eqref{eqn:laplacian-of-f}, $\cF$-invariant solutions $\phi$ of \eqref{eq:homsph} or, more generally,  \eqref{eq:extremal} satisfy the ODE 
\begin{equation}\label{eq:ode-s}
    \phi''-H \phi' + \lambda \phi = 0, \qquad \text{ in the connected set }\{\phi>0\} \subset [0, \tfrac{\pi}{g} ].\tag{$\star$}
\end{equation} 
Moreover, the boundary $\partial\{\phi>0\}$ consists of leaves of the foliation, and thus $|\nabla \phi| = |\phi'|$.
When $\{\phi>0\}$ is an interval $(s_1, s_2)$ compactly contained in $[0,\frac{\pi}{g}]$ (i.e. Type II solutions), the above equation is regular, and the boundary condition in \eqref{eq:homsph} or \eqref{eq:extremal} becomes 
\begin{equation}\label{eq:BC-2}
    \left\{\begin{array}{l}
     \phi(s_1) = \phi(s_2) = 0,  \\
     \phi'(s_1) = -\phi'(s_2) = 1.
\end{array} \right. 
\end{equation} 
When the interval $\{\phi>0\}$ contains either one of the endpoints (i.e. Type I solutions), since the above equation is singular in the first-order term at $0$ and $\frac{\pi}{g}$, the boundary condition is
\begin{equation}\label{eq:BC-1}
    \left\{\begin{array}{l}
     \phi(s_2)=0 \\
     \phi'(0) = 0, \phi'(s_2) = -1
\end{array} \right. \text{ if } s_1 = 0, \qquad
\left\{\begin{array}{l}
     \phi(s_1)=0 \\
     \phi'(s_1) = 1, \phi'( \tfrac{\pi}{g} ) =0
\end{array} \right. \text{ if } s_2 = \tfrac{\pi}{g}. 
\end{equation} 

The boundary conditions~\eqref{eq:BC-1} or~\eqref{eq:BC-2} correspond to the two types of solutions~\eqref{eqn:U-m-foliation} and~\eqref{eqn:U-bar-m-second-solution} from Theorem~\ref{thm:isoparametric-foliation-one-phase-cone}.
We call functions satisfying \eqref{eq:BC-1} Type I solutions, whose spherical positivity set $\Omega^S$ corresponds to a tubular neighborhood of a focal submanifold $M_1$, with free boundary a regular leaf of the isoparametric foliation.
We call functions satisfying~\eqref{eq:BC-2} Type II solutions, with $\Omega^S$ given by the spherical band between two regular leaves of the foliation and spherical free boundary expressed as the disjoint union of two regular leaves.

Geometrically, the involution $s \leftrightsquigarrow \frac{\pi}{g}-s$ corresponds to changing the isoparametric foliation $\{M_s\}_{s\in [0,\frac{\pi}{g}]}$ to the same foliation but in reversed order, namely $\{\tilde{M}_{\tau}\}_{\tau\in [0,\frac{\pi}{g}]}$ where $\tilde{M}_{\tau} = M_{\frac{\pi}{g}-\tau}$. In particular, we exchange the two focal submanifolds by $m_1 \leftrightsquigarrow m_2$.
It therefore suffices to consider Type I solutions on $[0,s_2] \subset [0,\frac{\pi}{g})$ satisfying the first pair of boundary conditions in \eqref{eq:BC-1}. 
Using this symmetry, when $m_1 = m_2$ and for all but countably many $\lambda$, we can also express the Type II solutions explicitly in terms of appropriate averages of the hypergeometric functions centered at the two focal submanifolds of the isoparametric foliation.

For isoparametric foliations with $g \in \{ 1,3,6 \}$ principal curvatures, we always have $m_1=m_2$, see \cites{Cartan1939, Abresch}; when $g=2$, we have $m_1=m_2$ when the isoparametric hypersurfaces are the product of two spheres of the same dimension.
Finally, $g=4$ produces the homogeneous isoparametric foliations with $m_1 = m_2 \in \{ 1,2 \}$.
In these cases, we have $H\left( \frac{\pi}{g} - s\right) =-H(s)$ by \eqref{eqn:h-sigma}, and thus the solution space of equation \eqref{eq:ode-s} is preserved by the involution $s \leftrightsquigarrow \frac{\pi}{g}-s$, meaning that $\phi(\frac{\pi}{g}-s)$ is a solution to the same equation whenever $\phi(s)$ is.
In particular, after we establish the uniqueness of Type II solutions, it follows by symmetry that $\{\phi>0\}$ contains the \emph{minimal} leaf $M_{\frac{\pi}{2g}}$.

\subsection{Geometry of the radial graphs}\label{sec:RadialGraph}

We next study the geometry of the \textit{radial graph} $\Sigma$ associated to the homogeneous solution $U$ in Definition~\ref{def:LinkOfCone}. By homogeneity, $\Sigma$ determines $U$, so analytic properties of $U$ are reflected in the geometry of $\Sigma$.
In particular, the number of ends of $\Sigma$ equals the number of connected components of $\partial \{ U>0\}\cap \bS^{n-1}$. 
This distinguishes the two families in Theorem~\ref{thm:isoparametric-foliation-one-phase-cone}: Type I examples $U_{M_i}$ have one end, while Type II examples $\bar U_M$ have two.
The radial graphs $\Sigma_{M_1},\Sigma_{M_2}$ of these homogeneous solutions can be topologically identified with the open normal neighborhoods of a focal submanifold $M_i \subset \bS^{n-1}$.
On the other hand, the radial graph $\bar{\Sigma}_{M}$ can be topologically identified with $M \times \bR$, the product of a regular leaf with a line.

Theorem~\ref{thm:unified-theorem} establishes a rigidity principle for homogeneous solutions of the Bernoulli free boundary problem: the symmetry of the radial graph determines the solution canonically.
The simplest instance is the axisymmetric case, where the only non-flat solution is the De Silva–Jerison cone. 
Its radial graph $\bar{\Sigma}_{\bS^{n-2}}$ is a two-ended hypersurface diffeomorphic to $\bR \times \bS^{n-2}$, invariant under rotations fixing an axis, and its intersections with spheres centered at the origin are unions of round spheres. 
In contrast, the flat solution has a one-ended radial graph isometric to Euclidean space. Thus, among axisymmetric homogeneous solutions, the number of ends already distinguishes the flat and De Silva–Jerison solutions. More generally, Theorem~\ref{thm:unified-theorem} shows that the uniqueness of the De Silva–Jerison cone is not an isolated consequence of rotational symmetry, but the first instance of a classification that applies to every prescribed isoparametric foliation.

The general geometric characterization of radial graphs is expressed through the distance function to the origin on $\Sigma_U$.
For $R>0$, let us write $\Sigma_U(R) := \Sigma_U \cap \bS^{n-1}(R)$.
If $U = \rho \phi(s)$ is $\cF$-invariant in the sense of~\cite{wang-on-a-class}, then each $\Sigma_U(R)$ is a union of dilates of leaves of the isoparametric foliation, possibly including a focal submanifold.
Conversely, if every $\Sigma_U(R)$ is a union of entire dilates of leaves of a fixed isoparametric foliation, then the radial graph representation and one-homogeneity imply that $U$ is constant along those leaves.
Consequently, $U = \rho \phi(s)$, and Theorem~\ref{thm:unified-theorem} implies that $U$ is one of $\{ U_{M_1}, U_{M_2}, \bar{U}_M\}$.
These three possibilities can be distinguished precisely from the radial geometry: the Type II solution $\bar{U}_M$ has two ends, whereas the Type I solutions have one end.
More precisely, the normalized locus $\frac{1}{\textup{dist}(\Sigma_U,0)} \Sigma_U( \textup{dist}(\Sigma_U,0))$ is a regular leaf for $\bar{U}_M$, and is the focal submanifold $M_i$ for $U_{M_i}$.
Thus, the radial foliation together with its distinguished minimum-radius locus identifies the homogeneous solution uniquely, including cases in which the two Type I solutions possess the same ambient symmetry group or have diffeomorphic radial graphs.

\begin{lemma}\label{lemma:star-shaped-1}
Let $\nu$ be the unit normal vector along the hypersurface $\Sigma \subset \bR^n$, pointing in the direction where $U$ is increasing.
    Then, $\la x, \nu \rg \geq 1$ everywhere along $\Sigma$, with strict inequality unless $U = \la a, x \rg_+$ is a half-plane solution.
    In particular, the hypersurface $\Sigma$ is a radial graph.
\end{lemma}
\begin{proof}
    Since $\Sigma$ is a level set of $U$, we have $\nu = \frac{\nabla U}{|\nabla U|}$.
    Using Euler's equality $\la x, \nabla U \rg = U$ for homogeneous functions, we then find $\la x, \nu \rg = \frac{\la x, \nabla U \rg}{|\nabla U|} = \frac{1}{|\nabla U|}$ due to $U=1$ on $\Sigma$.
    Moreover, $U$ is a harmonic function, hence $\Delta ( |\nabla U|^2) = 2 \, |D^2 U|^2 \geq 0$ means that $|\nabla U|$ is a homogeneous degree-$0$ subharmonic function whose maximum occurs on the boundary.
    The strong maximum principle now forces $|\nabla U| < 1$ away from $\partial \{ U > 0 \}$ unless $U = \la a, x \rg_+$ is a half-plane solution; therefore, $\la x, \nu \rg > 1$ as claimed. 
\end{proof}

While Lemma~\ref{lemma:star-shaped-1} is not used in the subsequent study of radial graphs, it highlights an important difference from the foliation of~\cite{hardt-simon-DSJ}: the inhomogeneous minimizers in the foliation of De Silva-Jerison-Shahgholian are also star-shaped, but their graphs satisfy $\la x, \nu \rg \to 0$ as $|x| \to \infty$.
On the other hand, the radial graphs $\Sigma_U$ are asymptotic to a vertically translated cone.

\begin{lemma}\label{lemma:euclidean-structure}
The following properties are equivalent.
\begin{enumerate}[(i)]
    \item The isometry group $\Gamma_U$ of the radial graph $\Sigma_U$ is the Euclidean group $\textup{E}(n-1)$.
    \item The radial graph $\Sigma_U$ is intrinsically flat.
    \item $U(x)$ is a half-plane solution $\la a,x \rg_+$ for some unit vector $a\in \mathbb{R}^n$.
\end{enumerate}
\end{lemma}
\begin{proof}
    Only one direction of the equivalence needs proof. Under the assumption
    \[ 
    \dim\Gamma_U = \dim E(n-1) = \tfrac{1}{2} n(n-1)
    \]
    Fix a point $x\in \Sigma_U$ and consider the stabilizer $\Gamma_x:= \{G\in \Gamma_U: G(x)=x \}$. Since $\Sigma_U$ is a connected $(n-1)$-manifold, $\Gamma_x$ is completely determined by its differential at $x$, which embeds into $O(n-1)$. 
    Thus, $\dim \Gamma_p \leq \dim O(n-1) = \frac{(n-1)(n-2)}{2}$, and the orbit $\Gamma_U \cdot x = \{G(x): G\in \Gamma_U \}$ has dimension
    \[ \dim(\Gamma_U \cdot x) = \dim \Gamma_U - \dim \Gamma_x \geq \tfrac{1}{2} n(n-1) - \tfrac{1}{2} (n-1)(n-2) = n-1.\]
    Connectedness of $\Sigma_U$ then implies that $\Gamma_U$ acts transitively.

    On the other hand, as $|x| \to \infty$ along $\Sigma_U$, its rescaling $\omega := \frac{x}{|x|}\in \Omega^S$ tends to $\partial\Omega^S$. 
    By the homogeneity of $U$, we have $|\nabla U(x)| = |\nabla U(\omega) | \to 1$; in particular, $|\nabla U(x)| \geq \frac{1}{2}$ for large $|x|$. 
    Let $\nu = \frac{\nabla U}{|\nabla U|}$ be the unit normal vector of the hypersurface $\Sigma_U$. 
    The shape operator of $\Sigma_U$ is given by
    \[ A_{\Sigma_U}(x) = -D\nu(x) = -\frac{D^2 U(x)}{|\nabla U(x)|},  \]
    and thus $|A_{\Sigma_U}(x)| = O(|x|^{-1})$ by the homogeneity of $D^2 U(x)$. Hence by the Gauss equation, the Riemannian curvature tensor satisfies $|\operatorname{Rm}_{\Sigma_U}(x)| = |A_{\Sigma_U}(x)|^2 = O(|x|^{-2})$. Since the isometry group $\Gamma_U$ acts transitively, it follows that $\operatorname{Rm}_{\Sigma_U} \equiv 0$, i.e. $\Sigma_U$ is intrinsically flat.

    Again by the Gauss equation, intrinsic flatness implies $\operatorname{rank} A_{\Sigma_U} \leq 1$.
    Hence $\operatorname{rank} D^2U(x) = \operatorname{rank}A_{\Sigma_U}(x) \leq 1$. 
    Because $U$ is harmonic, the symmetric matrix $D^2 U(x)$ is trace-free of rank $1$, so $D^2 U(x) \equiv 0$ on $\Sigma_U$, and hence on all of $\Omega$ by homogeneity. Finally $U(x) = \la a,  x\rg$ in $\Omega = \{x\in \mathbb{R}^n: \la a, x\rg >0\}$, and $|a| = 1$ by the Bernoulli boundary condition.  
\end{proof}

We recall a hypersurface rigidity lemma from~\cite{spivak}*{Chapter 12}, proved by Beez and Killing.
\begin{lemma}\label{lemma:shape-operators}
Consider connected hypersurfaces $\Sigma, \tilde{\Sigma} \subset \bR^n$, such that the second fundamental forms of $\Sigma, \tilde{\Sigma}$ have rank $\geq 3$ everywhere.
Then, every intrinsic isometry $F: (\Sigma, g) \to (\tilde{\Sigma}, \tilde{g})$ is the restriction of an ambient Euclidean isometry, namely there exist $Q \in O(n)$ and $b \in \bR^n$ such that $F(x) = Qx+b$ for every $x \in \Sigma$.
\end{lemma}
\begin{proof}
We regard $\Sigma_1$ as the abstract Riemannian manifold $M$ equipped with two isometric immersions $f: M \looparrowright \bR^n$ and $\tilde{f}: M \looparrowright \bR^n$ given by $f(x)=x$ and $\tilde{f}(x) = F(x)$.
Let $A$ and $\tilde{A}$ denote their shape operators.
Since the two immersions induce the same metric, their intrinsic curvature tensors agree.
Therefore, the Gauss equation gives
\[
\la AX,Z \rg \la AY,W \rg - \la AX, W \rg \la AY, Z \rg = \la \tilde{A}X, Z \rg \la \tilde{A}Y , W \rg - \la \tilde{A} X , W \rg \la \tilde{A}Y, Z \rg
\]
as both equal $R(X,Y,Z,W)$.
Thus, the Gauss tensors $A \wedge A = \tilde{A} \wedge \tilde{A}$ agree.
Since $A, \tilde{A}$ have rank $\geq 3$, this implies $\tilde{A} = A$ after a consistent choice of normal vectors.
Thus, $f, \tilde{f}$ have the same first and second fundamental forms, so they differ locally by an ambient Euclidean isometry, namely every $p \in M$ has a neighborhood $U_p \ni p$ and a rigid motion $T_p(x) = Q_p x + b_p$ such that $\tilde{f}|_{U_p} = (T_p \circ f)|_{U_p}$.
Because $M$ is connected, these local rigid motions assemble to a single ambient rigid motion $T(x) = Qx+b$ such that $F(x) = Qx+b$ on $\Sigma$.
This proves our assertion.
\end{proof}

\begin{proposition}\label{prop:intrinsic-ambient-isometry}
    Let $U, V$ be two non-flat homogeneous solutions in $\mathbb{R}^n$ with $n\geq 4$ that are constant along the leaves of isoparametric foliations $\cF_U$ and $\cF_V$, respectively.
    Then, every intrinsic isometry $F: (\Sigma_U, g_{\Sigma_U}) \to (\Sigma_V, g_{\Sigma_V})$ is the restriction of an ambient isometry $Q\in O(n)$ that preserves the origin, and $V(Qx) = U(x)$ for every $x\in \mathbb{R}^n$.

    In particular, the intrinsic isometry group $\Gamma_U$ agrees with the restriction to $\Sigma_U$ of the ambient symmetry group $\{Q\in O(n): U(Qx) = U(x) \text{ for every } x\in \mathbb{R}^n\}$.
\end{proposition}
\begin{remark}
The dimensional restriction $n \geq 4$ is necessary: for example, the catenoid and helicoid in $\mathbb{R}^3$ are isometric but not congruent. 
However, the only isoparametric foliation in $\mathbb{R}^3$ has $g=1$, and the corresponding non-flat homogeneous solutions are congruent to the De Silva-Jerison cone. 
\end{remark}
\begin{proof}
By the spherical Faber-Krahn inequality, the positivity set $\{ U > 0 \} \cap \bS^{n-1}$ is connected, so the $\cF_U$-invariance implies that $\partial \{ U > 0\} \cap \bS^{n-1}$ consists of at most two connected components, so $\Sigma_U$ has at most two ends.
Let $L\subset \mathbb{S}^{n-1}$ be the locus of points of $\Sigma_U$ with minimum distance locus to the origin.
For brevity, let $\ell := n-1- \dim L$.
If $\Sigma_U$ has one end, then $L \subset \bS^{n-1} ( \textup{dist}(\Sigma_U,0))$ is one of the focal submanifolds $M_1,M_2$ and $\ell = m_i + 1$; if $\Sigma_U$ has two ends, then $L$ is a regular leaf of the isoperimetric foliation and $\ell = 1$.
    
We fix a point $\omega_0 \in L$ in what follows, so $\nabla_{\mathbb{S}^{n-1}} U(\omega_0) = 0$.
Moreover, the spherical Hessian $D^2_{\mathbb{S}^{n-1}} U(\omega_0)$ vanishes along $T_{\omega_0} L$ and along the mixed tangential-normal directions. 
Observe that
\[
D^2_{\bS^{n-1}} U(\omega_0)|_{T^{\perp}_{\omega_0} L} = c \cdot \textup{id}_{T_{\omega_0}^{\perp} L}, \qquad \text{for some constant } \; c \neq 0.
\]
If $\Sigma_U$ has two ends, then $\dim T^{\perp}_{\omega_0} L = \ell = 1$, so this property is automatic.
When $\Sigma_U$ has two ends, the $\cF_U$-invariance implies that the value of $U$ depends only on the spherical distance to the focal submanifolds, so $D^2_{\bS^{n-1}} U(\omega_0)$ takes the form of a block matrix $\begin{psmallmatrix} 0 & 0 \\ 0 & \lambda \, \textup{Id}_{\ell} \end{psmallmatrix}$, for some $\lambda \in \bR$.
Then,
\[
\textup{tr} \, D^2_{\bS^{n-1}} U(\omega_0) = \Delta_{\bS^{n-1}} U(\omega_0) = - (n-1) U(\omega_0),
\]
which implies that $\lambda = - \frac{n-1}{\ell} U(\omega_0) \neq 0$.
At the corresponding $x_0 := \frac{\omega_0}{U(\omega_0)} \in \Sigma_U$, we compute
\[
A_{\Sigma_U}(x_0) = - ( D^2_{\bS^{n-1}} U(\omega_0) + U(\omega_0) \, \textup{Id}_{n-1}) 
\]
so $A_{\Sigma_U}(x_0)$ has full rank $n-1 \geq 3$.
By the same argument for $\Sigma_V$, $A_{\Sigma_V}$ has full rank $n-1$ everywhere.
Therefore, Lemma~\ref{lemma:shape-operators} implies that $F: (\Sigma_U, g_U) \to (\Sigma_V , g_V)$ is the restriction of an ambient Euclidean isometry $x \mapsto Qx+b$ near $x_0$.
The hypersurfaces $\Sigma_U, \Sigma_V$ are real-analytic and connected, so unique continuation implies that $F(x)=Qx+b$ everywhere, and thus $F$ is surjective.

It remains to show $b=0$; we first prove that $Q (\Sigma_U) = \Sigma_V$. 
Let us write $\Omega^S_U = \{ U > 0 \} \cap \bS^{n-1}$ and $\Omega^S_V = \{ V>0 \} \cap \bS^{n-1}$.
We recall that as $|x| \to \infty$ along $\Sigma_V$, the point $\frac{x}{|x|}$ tends to the boundary $\partial \Omega^S_V$.
Moreover, any $\omega\in \partial \Omega_U^S$ is the limit of a sequence $\omega_j \in \Omega_U^S$, and the corresponding sequence $x_j:= \frac{\omega_j}{U(\omega_j)}$ satisfies $x_j \in \Sigma_U$ and $|x_j| \to \infty$.
Since $\frac{Qx_j+b}{|Qx_j+b|} - \frac{Qx_j}{|Qx_j|} \to 0$ for $|x_j| \to \infty$, we obtain $\lim_{j \to \infty} Q \omega_j = \lim_{j \to \infty} \frac{F(x_j)}{|F(x_j)|} \in \partial \Omega^S_V$, and hence $Q (\partial\Omega_U^S) \subset \partial\Omega_V^S$. 
Using the surjectivity of $F$ and a symmetric argument, we deduce the reverse inclusion, and hence $\partial Q(\Omega^S_U) = \partial \Omega^S_V$.
On the other hand, the two spherical domains $Q (\Omega_U^S)$ and $\Omega_V^S$ are connected, have the same first Dirichlet eigenvalue $(n-1)$, and are not hemispheres, so by the Faber-Krahn inequality, they have volume greater than $\frac{1}{2} \mathcal{H}^{n-1}(\mathbb{S}^{n-1})$. 
Thus, they must intersect, and since their boundaries agree, we deduce that $Q(\Omega^S_U) = \Omega^S_V$.
The homogeneity now implies that $U(Qx) = V(x)$, so $Q(\Sigma_U) = \Sigma_V$ as claimed.

Finally, using $\Sigma_V = F(\Sigma_U) = Q (\Sigma_U) + b$, we deduce that $\Sigma_V + b = \Sigma_V$.
Recall from Lemma~\ref{lemma:euclidean-structure} that $|A_{\Sigma_V}(x)| = O ( |x|^{-1})$.
Hence, if $b \neq 0$, we have $|A_{\Sigma_V}(x)| = |A_{\Sigma_V}(x+mb)| = O(m^{-1}) \to 0$ as $m \to \infty$ for every fixed $x \in \Sigma_V$, so $\Sigma_V$ is a flat plane and $V$ is the half-plane solution, contradicting our assumption.
Therefore, $b=0$ and $V(Qx) = U(x)$ as desired.
\end{proof}

We now study the induced metric $g_{\Sigma}$ on the radial graph in terms of the geometry of the isoparametric foliation.
Away from the focal submanifolds, the normal parametrization of the sphere gives $g_{\bS^{n-1}} = ds^2 + g_{M_s}$, for $g_{M_s}$ the induced metric on the regular leaf $M_s$.
Consequently, the Euclidean metric assumes the form $g_{\textup{Euc}} = d \rho^2 + \rho^2 ( ds^2 + g_{M_s})$.
For an $\cF$-invariant homogeneous solution $U = \rho \phi(s)$, the radial graph $\Sigma_U$ can be parametrized by $\Sigma_U = \{ r(s) \omega : \omega \in M_s \}$, where $r(s) = \frac{1}{\phi(s)}$.
Since the radial direction is orthogonal to the tangent spaces of the leaves, we compute
\begin{align*}
g_{\Sigma} &\;= (r'(s)^2 + r(s)^2) \, ds^2 + r(s)^2 \, g_{M_s} = d \tau^2 + r(\tau)^2 g_{M_{s(\tau)}}, \\
\tau(s)  &:= \int_{s_0}^s \sqrt{r'(\sigma)^2 + r(\sigma)^2} \, d\sigma.
\end{align*}
We call $\tau(s)$ the arc-length parameter.
Thus, the geometry of the radial graph is determined by the radial profile r and the family of metrics $g_{M_s}$ induced on the isoparametric leaves. 

We next describe this expression in the axisymmetric and Clifford cases, where the geometry of the leaves gives rise to explicit warped-product metrics.
\begin{enumerate}[(i)]
    \item In the axisymmetric case $g=1$, the regular leaves are round spheres with $g_{M_s} = \sin^2 s \, g_{\bS^{n-2}}$.
    Writing $t = \cos s$ expresses the spherical metric as
    \[
    g_{\bS^{n-1}} = \tfrac{1}{1-t^2} dt^2 + (1-t^2) \, g_{\bS^{n-2}}.
    \]
    For the De Silva-Jerison solution $U = c \rho f(t)$, whose positivity interval is $(-t_n, t_n)$, the radial graph is parametrized by $r(t) = \frac{1}{c f(t)}$.
    Hence, the induced metric on the graph becomes
    \begin{align*}
    g_{\Sigma_{\textup{DSJ}}} &= \frac{f(t)^2 + (1-t^2) f'(t)^2}{c^2 f(t)^4 (1-t^2)} \, dt^2 + \frac{1-t^2}{c^2 f(t)^2} \, g_{\bS^{n-2}} = d \tau^2 + a(\tau)^2 \, g_{\bS^{n-2}},
\end{align*}
where $a(\tau) := \frac{\sqrt{1- \tau(t)^2}}{c \, f(\tau(t))}$ and we introduced the arc-length parameter
\[
\tau(t) := \int_0^t \frac{\sqrt{f(\xi)^2 + (1-\xi^2) f'(\xi)^2}}{ c \, f(\xi)^2 \sqrt{1-\xi^2}} \, d \xi.
\]
Therefore, $g_{\Sigma_{\textup{DSJ}}}$ admits a warped product representation.
The radial graph is two-ended, and its warping function is positive everywhere, attaining its minimum at the central leaf $t=0$. 

In contrast, the radial graph of a flat solution is an affine hyperplane, whose metric in geodesic polar coordinates centered at the point closest to the origin is $g_{\Sigma} = d \tau^2 + \tau^2 \, g_{\bS^{n-2}}$, $\tau \geq 0$.
Hence, the flat graph has a collapsing spherical factor, whereas the De Silva–Jerison graph has a non-collapsing central sphere and extends to infinity in both directions. This distinction is illustrated in Figure~\ref{fig:DSJSigma}.

\item In the Clifford case $g = 2$, the leaves $M_s = \bS^p (\sin s) \times \bS^q(\cos s)$ are products of rescaled hypertori with $p+q=n-2$, so $g_\Sigma$ becomes a doubly warped product over $\bS^p \times \bS^q$, as
\begin{align*}
    g_{M_s} = \sin^2 s \, g_{\bS^p} + \cos^2 s \, g_{\bS^q}, \qquad
    g_{\Sigma} = d \tau^2 + a_1(\tau)^2 \, g_{\bS^p} + a_2(\tau)^2 \, g_{\bS^q},
\end{align*}
where $a_1(\tau) = r(\tau) \sin s(\tau)$ and $\quad a_2(\tau) = r(\tau) \cos s(\tau)$.
The behavior of the warping functions distinguishes the three solutions.
For $U_{M_1}$, the first factor collapses at the focal manifold $M_1$, while the second remains positive. 
For $U_{M_2}$, the opposite occurs.
Both radial graphs have one end, but the collapsing factor identifies the focal submanifold in the minimum-distance locus.

For the Type II solution $\bar{U}_M$, neither factor collapses; both warping functions remain positive, and the radial graph has two ends, with its minimum-radius locus corresponding to a regular Clifford hypersurface. 
In this case, the positivity set of $\bar{U}_M$ contains only regular leaves.
The three possibilities for the warping functions $a_j(\tau)$ are illustrated in Figure~\ref{fig:g2WarpingFunctions}.
\end{enumerate}

\begin{figure}
    \centering
    \includegraphics[width=0.4\linewidth]{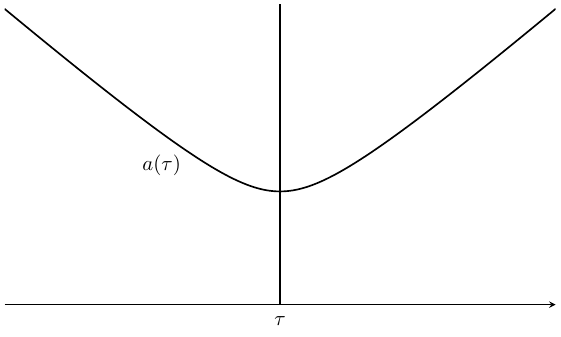}
    \caption{The warping function $a(\tau)$ describing the geometry of the radial graph $\Sigma_{\bar{U}_{\bS^{n-2}}}$ for the De Silva-Jerison cone in $n = 7$.
    Since this cone does not contain the focal manifold, it is of Type II and this warping function never vanishes.
    From Theorem~\ref{thm:unified-theorem}, the Type I case with this symmetry is the Euclidean metric where $a(\tau) = \tau$. 
    }
    \label{fig:DSJSigma}
\end{figure}

\begin{figure}
    \centering
    \includegraphics[width=0.3\linewidth]{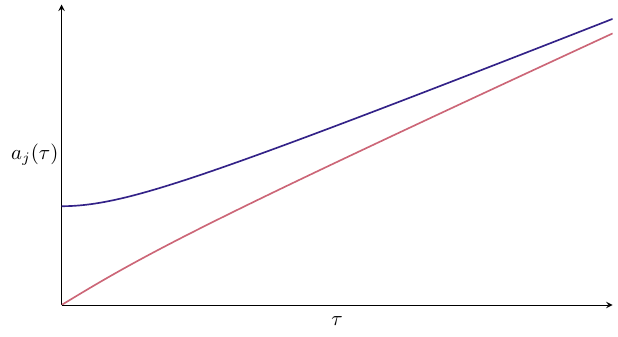}
    \includegraphics[width=0.3\linewidth]{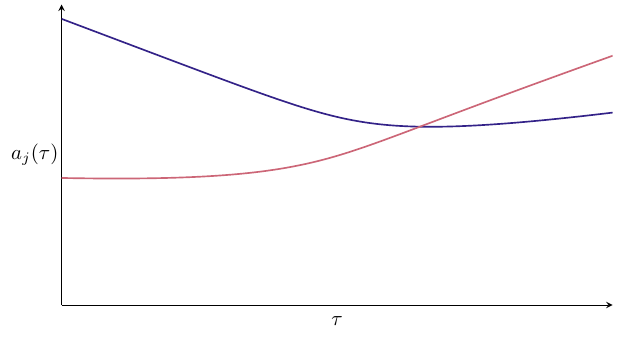}
    \includegraphics[width=0.3\linewidth]{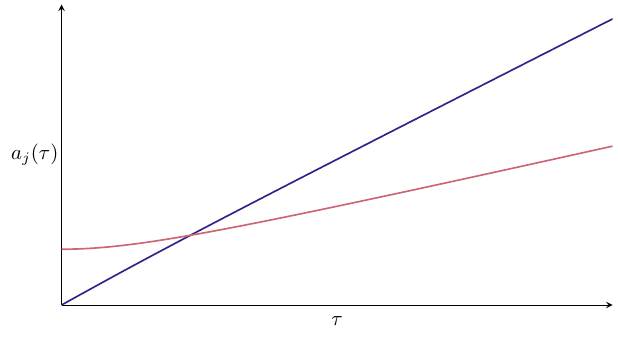}
    \caption{The warping functions $a_j(\tau)$ from metric for $\Sigma$, corresponding to the Clifford foliation by $\bS^4 \times \bS^1 \subset \bS^6$.
    The outer graphs display the two Type I examples, where the $y$-axis represents the focal manifold where $\bS^p$ or $\bS^q$ collapses.
    The middle graph shows the warping functions for the Type II curve, where neither vanishes. 
    }
    \label{fig:g2WarpingFunctions}
\end{figure}

Finally, for $g \geq 3$, the induced metric $g_{M_s}$ does not admit a multiple-warped product representation because the decomposition of $TM_s$ into the sum of the $g$ distributions corresponding to each principal curvature does not arise from a product decomposition of $M_s$.
We obtain the following description.

Let $M_{s_0} \subset \bS^{n-1}$ be the unique minimal leaf of the isoparametric foliation $\{ M_s \}$ and let $\Phi_s : M_{s_0} \to M_s$ be the diffeomorphism given by the normal parallel map, which is expressible as
\begin{equation}\label{eqn:Phi-s(x)}
\Phi_s(x) = \cos (s - s_0) \, \Phi_{s_0}(x) + \sin (s - s_0) \, \nu(x) 
\end{equation}
for $\Phi_{s_0} : M_{s_0} \to \bS^{n-1}$ the embedding map. 
Let $A_0$ be the shape operator of the embedding $M_{s_0} \hookrightarrow \bS^{n-1}$ and let $E_j \subset TM_{s_0}$ denote the principal distributions of the tangent bundle consisting of all tangent directions with the same principal curvature $\lambda_j(s_0)$, where $\lambda_j(s) = - \cot ( s_0 + \frac{(j-1) \pi}{g})$ with the convention $d \nu = - d \Phi_{s_0} \circ A_0$; see, for example, the computation in~\cite{cecilRyan}.
Differentiating the equation~\eqref{eqn:Phi-s(x)} in some tangent direction $X \in T_x M_{s_0}$ and using the Weingarten formula $d \nu(X) = - d \Phi_{s_0}(A_0) X)$, we find
\begin{align*}
    d \Phi_s|_x(X) &= \cos(s-s_0) \, d \Phi_{s_0}(X) + \sin(s-s_0) \, d \nu(X) \\
    &= d \Phi_{s_0} \bigl[ (\cos(s-s_0) \, I - \sin (s-s_0) A_0 ) \, X \bigr] = \frac{\sin \bigl(s + \frac{(j-1) \pi}{g} \bigr)}{\sin \bigl( s_0 + \frac{(j-1) \pi}{g} \bigr)} \, d \Phi_{s_0}(X).
\end{align*}
We write $g_j$ for the restriction of the induced metric to $E_j$, so using the fact that the principal distributions $E_j$ are eigenspaces of the Weingarten map and mutually orthogonal, we obtain
\[
\Phi^*_s g_{M_s} = \sum_{j=1}^g \left( \frac{\sin \bigl(s + \frac{(j-1) \pi}{g} \bigr)}{\sin \bigl( s_0 + \frac{(j-1) \pi}{g} \bigr)} \right)^2 g_j.
\]
The induced metric on the radial graph is therefore
\[
g_\Sigma = d\tau^2 + r(\tau)^2 \sum_{j=1}^g \left(\frac{\sin\bigl(s+\frac{(j-1)\pi}{g}\bigr)}{\sin\bigl(s_0+\frac{(j-1)\pi}{g}\bigr)}\right)^2g_j,
\]
for $(\tau, r(\tau))$ the arc-length parameter and radial coordinate of $\Sigma$.
Notice that $g_{\Sigma}$ is diagonal with respect to the principal curvature decomposition; for $g=2$, it reduces to the doubly warped product metric obtained above, while for $g \geq 3$, it does not integrate to a genuine product metric.

\begin{remark}
The preceding classification, together with Theorem~\ref{thm:unified-theorem}, distinguishes two levels of rigidity.
The prescribed isoparametric foliation, together with the radial geometry of $\Sigma$, determines the homogeneous solution. For homogeneous foliations, the same conclusion follows from invariance under the corresponding ambient cohomogeneity-one action. By contrast, the abstract isometry group of $\Sigma_U$ need not by itself distinguish the solutions associated with different focal submanifolds.
Indeed, for $g = 2$, every block-form orthogonal transformation preserves the profile, so both $U_{M_1}$ and $U_{M_2}$ have the same isometry group $\Gamma_{U_{M_1}}^0 = \Gamma^0_{U_{M_2}} = \textup{SO}(p+1) \times \textup{SO}(q+1)$; however, $\Sigma_{M_1} \cong \bS^p \times \bR^{q+1}$ and $\Sigma_{M_2} \cong \bR^{p+1} \times \bS^q$ are not even homotopy equivalent when $p \neq q$. 
These are distinguishable by the topology, but not the algebra. 

The more sophisticated foliations when $g > 2$ yield homogeneous solutions with distinct behaviors that are not distinguishable only by the topology.
The case of $(g,m_1,m_2) = (6,1,1)$ has focal submanifolds $M_1, M_2$ diffeomorphic to $\bR\bP^2 \times \bS^3$, but not isometric to each other.
Consequently, the radial graphs $\Sigma_{M_1} \cong \Sigma_{M_2}$ are diffeomorphic but not isometric, else Proposition \ref{prop:intrinsic-ambient-isometry} would force their minimum-distance loci $M_1, M_2$ to be isometric. 
This example exhibits two solutions $U_{M_1}$ and $U_{M_2}$ being genuinely distinct, despite having the same profile curves.
\end{remark}

\section{Existence and uniqueness of solutions}\label{section:one-phase}
We now perform the ODE analysis to show the existence and uniqueness of solutions of both types.
While the Type I solutions $U_{M_1}$ and $U_{M_2}$ are determined uniquely via the hypergeometric functions $f_{M_1}$ and $f_{M_2}$, the Type II solution $\bar{U}_M$ arises via a shooting-continuity argument, so its uniqueness is more complicated to establish.
This situation is analogous to the corresponding problem for minimal surfaces and self-shrinkers of mean curvature flows, where Brendle's remarkable paper~\cite{clifford-tori} proved the Lawson conjecture on the uniqueness of the Clifford torus in $\bS^3$, later extended to CMC tori and Weingarten tori~\cites{weingarten-tori , andrews-li-cmc-tori}.
The uniqueness of the Angenent torus~\cite{angenent-tori} remains widely open, as well as the minimal surface counterparts of the $\bar{U}_M$ constructed in~\cites{new-minimal-surfaces , capillary-FBP}.
The uniqueness of the axisymmetric capillary cones of~\cite{FTW-1} was proved in~\cite{new-minimal-surfaces}.

To write \textbf{Type I solutions} explicitly, we perform a change of variables to relate solutions to \eqref{eq:ode-s} with hypergeometric functions. When $g=1$, sometimes a different change of variables is used, because it is more natural to center the problem at the equator, instead of the two focal submanifolds. 
A direct change of variables $t := \sin \frac{gs}{2} \in [0,1]$ shows that a function $\phi(s)$ satisfies the equation~\eqref{eq:ode-s}, with $H = \frac{g}{2} (m_2 \tan \frac{g s}{2} - m_1 \cot \frac{g s}{2} )$, if and only if $f(t) := \phi(s)$ is in the kernel of the operator
    \begin{equation}\label{eqn:self-adjoint-legendre}\tag{L}
\begin{split}
    \cL_{M} f &:= (1-t^2) f'' + \bigl( \tfrac{m_1}{t} - (m_1 + m_2 + 1) t \bigr) f' + \tfrac{4\lambda}{g^2} f. \\
\end{split}
\end{equation}
Similarly, for $g=1$, we consider $t := \cos s \in [-1,1]$ to obtain $f \in \ker \cL_{\bS^{n-2}}$, where
\[
\cL_{\bS^{n-2}} f := (1-t^2) f'' - (n-1) tf' + \lambda f.
\]

We now define the auxiliary parameter
\[ 
\nu= \frac{1}{2} \left( \sqrt{\left( \frac{m_1+m_2}{2}\right)^2 + \frac{4\lambda}{g^2}}-\frac{m_1+m_2}{2} \right).
\]
In particular, $\nu = \frac{1}{g}$ when $\lambda=n-1$.
We define the function
\begin{equation}\label{eqn:Phi-M1}
    \Phi_{M_1}(s) = {}_2 F_1 \Bigl( \nu + \frac{m_1+m_2}{2} , - \nu; \frac{m_1+1}{2} ; \sin^2 \frac{gs}{2} \Bigr),
\end{equation}
given by the hypergeometric function solving \eqref{eqn:self-adjoint-legendre} that is regular at $0$. 
Hence, $\Phi_{M_1}$ is a solution of the ODE~\eqref{eq:ode-s}, and is the normalized regular solution at $s=0$; moreover, it is the unique Type I profile, up to scaling, which is regular at $0$. 
Let $s_{M_1} \in (0,\frac{\pi}{g})$ denote its first zero, whose existence is guaranteed by the asymptotic properties of hypergeometric functions near the endpoints $z=0$ and $z=1$.
Scaling by $\frac{1}{|\Phi'_{M_1}(s_{M_1})|}$ satisfies the Neumann boundary condition and produces the required solution of \eqref{eq:homsph} and \eqref{eq:extremal}.
Thus in particular, $U_{M_1}(\rho, \omega) = \frac{1}{|\Phi'_{M_1}(s_{M_1})|} \rho \Phi_{M_1}(s\omega)$ is a homogeneous solution of the Bernoulli free boundary problem.
Similarly, let $\Phi_{M_2}(s) = {}_2 F_1( \nu + \frac{m_1+m_2}{2} , - \nu; \frac{m_2+1}{2} ; \cos^2 \frac{gs}{2})$, and the case of $M_2$ follows directly under the change of variables $s \mapsto \frac{\pi}{g}-s$ and $m_1 \leftrightsquigarrow m_2$.
These are precisely the solutions~\eqref{eqn:U-m-foliation} in Theorem~\ref{thm:isoparametric-foliation-one-phase-cone}.

For all but countably many values of $\lambda$, the hypergeometric function solving \eqref{eqn:self-adjoint-legendre} which is regular at one endpoint blows up at the other endpoint, and they two form a pair of fundamental solutions to the ODE and define the two Type I solutions $\Phi_{M_1}, \Phi_{M_2}$. However, when $\lambda=g^2k(k+\frac{m_1+m_2}{2})$, $k\in \mathbb{N}$, these two hypergeometric functions coincide and become a polynomial, and thus it is regular at both endpoints. The latter is the case when $\lambda=n-1$ and $g=1$, where a pair of fundamental solution to \eqref{eqn:self-adjoint-legendre} is given by the hypergeometric function $f(t)=t$ and a (even) singular solution that gives rise to the De Silva-Jerison cone.

The construction of \textbf{Type II solutions}, on the other hand, is by a shooting-continuity argument.
\begin{lemma}\label{lemma:uniqueness-solution}
    Let $H \in C^1(a,b)$ satisfy $H'>0$, and $\lambda>0$.
    For $\sigma$ in an interval $I \subset (a,b)$, let $\phi_{\sigma}$ solve the initial value problem 
    \[
    \phi_{\sigma}'' - H \phi'_{\sigma} + \lambda \phi_{\sigma} = 0, \qquad \phi_{\sigma}(\sigma) = 0, \quad \phi'_{\sigma}(\sigma) = 1.
    \]
    Suppose that, for every $\sigma \in I$, $\phi_{\sigma}$ has a second zero $\tau(\sigma)> \sigma$, is positive on $(\sigma, \tau(\sigma))$. Then
    \begin{enumerate}[(i)]
        \item It satisfies $(\phi'_{\sigma})^2 - \phi_{\sigma} \phi''_{\sigma} >  0$ on this interval, i.e. $\phi_\sigma$ is $\log$-concave.
        \item The maps $\sigma \mapsto \tau(\sigma)$ and $\sigma \mapsto - \phi_{\sigma}'(\tau(\sigma))$ are strictly increasing on $I$.
    \end{enumerate} 
\end{lemma}
\begin{remark}
Under the stronger assumption $H'> \lambda$, the solution $\phi_\sigma$ is strictly concave, because $( m \phi''_{\sigma})' = m ( H' - \lambda) \phi'_{\sigma}$ implies that $( m \phi''_{\sigma})'$ and $\phi'_{\sigma}$ retain the same sign.
For the Bernoulli free boundary problem, where $\lambda = n-1$, we indeed have $H' > n-1$, due to the sharp inequality
    \begin{align*}
    H'(s) = \tfrac{1}{4} g^2  (m_2 \sec^2 \tfrac{gs}{2} + m_1 \csc^2 \tfrac{gs}{2}) &= \tfrac{1}{4} g^2 (m_1 + m_2 + m_2 \tan^2 \tfrac{gs}{2} + m_1 \cot^2 \tfrac{gs}{2}) \\
    &\geq \tfrac{1}{4} g^2 (m_1 + m_2 + 2 \sqrt{m_1 m_2}). 
\end{align*}
Since $n-1 = 1 + \frac{g}{2} (m_1+m_2)$, we obtain
\[
H' - (n-1) \geq \tfrac{1}{4} g^2 ( \sqrt{m_1} + \sqrt{m_2})^2 - 1 - \tfrac{1}{2} g(m_1+m_2) = \tfrac{1}{4} g(g-2) (m_1+m_2) + \tfrac{1}{2} g^2 \sqrt{m_1m_2} - 1.
\]
This expression is strictly positive for $g \geq 2$ and $m_1 , m_2 \geq 1$, hence $H'(s) > n-1$ for every $s\in (0 , \frac{\pi}{g})$.
\end{remark}
\begin{proof}
    Let $q(s) := \exp ( - \int_{s_0}^s H(r) \, dr)$, so the equation becomes $(q \phi'_\sigma)' =- \lambda q \phi_\sigma < 0$, and thus $q \phi'_\sigma$ is decreasing. 
    Since at the endpoints $\phi'_\sigma(\sigma)>0, \phi'_\sigma(\tau(\sigma))<0$, it follows that $\phi_\sigma$ has a unique critical point at $c \in (\sigma, \tau(\sigma))$; moreover $\phi'_\sigma > 0$ before $c$, and $\phi'_\sigma<0$ after $c$. 
    Thus, $c$ is a local maximum, and $\phi''_{\sigma}(c)<0$.
    Let $\psi:= (\phi'_{\sigma})^2 - \phi_{\sigma} \phi''_{\sigma}$, for which we compute $\psi' - H \psi = - H' \phi \phi'$ and
    \[
    (q\psi)' = q' \psi + q \psi' = q (-H \psi + \psi') = - q H' \phi \phi'.
    \]
    Hence, $m\psi$ is strictly decreasing before $c$, and strictly increasing after $c$. 
    In particular, $\psi(c) = (\phi'_{\sigma}(c))^2 - \phi_{\sigma}(c) \phi''_{\sigma}(c) > 0$ implies that $\psi(s) \geq \frac{(q \psi)(c)}{q(s)} > 0$ for $s>c$.

    We define the energy $E_{\sigma}$ and a continuous phase $\theta(s)$ by
    \[
E_{\sigma} := (\phi'_{\sigma})^2 + \lambda \phi^2_{\sigma}, \qquad \phi'_{\sigma} = \sqrt{E_{\sigma}} \, \cos \theta, \qquad \sqrt{\lambda} \, \phi_{\sigma} = \sqrt{E_{\sigma}} \, \sin \theta, \qquad \theta(\sigma) = 0.
\]
The definition of $\theta$ implies that $E'_{\sigma} = 2 H(s) E_{\sigma} \cos^2 \theta$ and
\[
\theta' = \sqrt{\lambda} E_{\sigma}^{-1} [ ( \phi'_{\sigma})^2 - \phi_{\sigma} \phi''_{\sigma}] = \sqrt{\lambda} - \tfrac{1}{2} H(s) \sin ( 2 \theta)
\]
where the expression is strictly positive on $(\sigma, \tau(\sigma))$, by the assumption on $\phi_{\sigma}$.
Thus, $\theta$ increases from $0$ to $\pi$ as $s$ increases from $\sigma$ to $\tau(\sigma)$.
This allows us to write $s = s_{\sigma}(\theta)$, whereby
\begin{equation}\label{eqn:phase-general-equation}
    \frac{d s_{\sigma}}{d \theta} = \frac{1}{\sqrt{\lambda} - \frac{1}{2} H(s_{\sigma}) \sin ( 2\theta)}, \qquad  \frac{d}{d \theta} \log E_{\sigma} (s_{\sigma}(\theta)) = \frac{2 H(s_{\sigma}) \cos^2 \theta}{\sqrt{\lambda} - \frac{1}{2} H(s_{\sigma}) \sin(2 \theta)}.
\end{equation}
    The common denominator is positive along every trajectory $\theta \mapsto s_{\sigma}(\theta)$ because it equals $\theta'$.
    Since $E_{\sigma}(\sigma) = 1$ and $E_{\sigma}(\tau(\sigma)) = \phi'_{\sigma} ( \tau(\sigma))^2$, integrating the second equation~\eqref{eqn:phase-general-equation} produces
    \begin{equation}\label{eqn:log-phi-prime}
    \log \phi'_{\sigma} (\tau(\sigma))^2 = \int_0^{\pi} \frac{2 H(s_{\sigma}(\theta)) \cos^2\theta}{\sqrt{\lambda} - \frac{1}{2} H(s_{\sigma}(\theta)) \sin(2 \theta)} \, d\theta.
    \end{equation}
    By the standard continuous dependence for ODEs, the simple zero $\tau(\sigma)$ depends smoothly on $\sigma$, hence so does $\sigma \mapsto \phi'_{\sigma}(\tau(\sigma))$.
    For $\sigma \in I$, we let $J_{\sigma}(\theta) := \frac{\partial s_{\sigma}}{\partial \sigma}(\theta)$, so differentiating~\eqref{eqn:phase-general-equation} produces
    \begin{align*}
    J'_{\sigma} (\theta) &= \frac{\frac{1}{2} H'(s_{\sigma}(\theta)) \sin(2 \theta)}{( \sqrt{\lambda} - \frac{1}{2} H(s_{\sigma}(\theta)) \sin ( 2 \theta))^2} J_{\sigma}(\theta), \qquad J_{\sigma}(0) = 1, \\
    &\frac{\partial}{\partial s} \Bigl( \frac{2 H(s) \cos^2 \theta}{\sqrt{\lambda} - \frac{1}{2} H(s) \sin (2 \theta)} \Bigr) = \frac{2 \sqrt{\lambda} H'(s) \cos^2 \theta}{ ( \sqrt{\lambda} - \frac{1}{2} H(s) \sin (2 \theta) )^2} > 0 \quad \text{for } \; \theta \neq \frac{\pi}{2}.
    \end{align*}
    In particular, $J_{\sigma}(\theta) > 0$ for every $\theta \in [0,\pi]$.
    Therefore, $\tau'(\sigma) = J_{\sigma} ( \pi) > 0$ implies that the map $\sigma \mapsto \tau(\sigma)$ is strictly increasing.
    As a result, differentiating the relation~\eqref{eqn:log-phi-prime} gives
    \[
    \frac{d}{d \sigma} \log \phi'_{\sigma}(\tau(\sigma))^2 = \int_0^{\pi} \frac{2 \sqrt{\lambda}  H'(s_{\sigma}(\theta)) \cos^2 \theta}{( \sqrt{\lambda} - \frac{1}{2} H(s_{\sigma}(\theta)) \sin ( 2 \theta))^2} J_{\sigma}(\theta) \, d \theta
    \]
    where the integrand is non-negative and strictly positive except at $\theta = \frac{\pi}{2}$.
    Then, $\frac{d}{d \sigma} \log \phi'_{\sigma}(\tau(\sigma))^2 > 0$ implies that $\sigma \mapsto - \phi'_{\sigma}(\tau(\sigma))$ is a strictly increasing function on $I$.
\end{proof}

\begin{proposition}\label{prop:general-uniqueness-argument}
Consider a function $H \in C^1(a,b)$ with $H'(s) > 0$ and suppose that
\begin{equation}\label{eqn:blowup-condition}
\int_a^{s_0} \exp \Bigl( \int_{s_0}^s H(r) \, dr \Bigr) \, ds = \infty, \qquad \int_{s_0}^b \exp \Bigl( \int_{s_0}^s H(r) \, dr \Bigr) \, ds = \infty
\end{equation}
for some $s_0 \in (a,b)$. Let $\lambda>0$.
For $\sigma \in (a,b)$, let $\phi_{\sigma}$ denote the solution of equation
    \[
    \phi_{\sigma}'' - H \phi'_{\sigma} + \lambda \phi_{\sigma} = 0, \qquad \phi_{\sigma}(\sigma) = 0, \quad \phi'_{\sigma}(\sigma) = 1.
    \]
There is a unique $\sigma_* \in (a,b)$ such that $\phi_{\sigma_*}$ reaches a second zero at $\tau(\sigma_*)$, with $\phi'_{\sigma_*} (\tau(\sigma_*)) = -1$.
\end{proposition}
\begin{proof}
By translating the independent variable, it suffices to consider $(a,b) = (0,L)$.
Let $q(s) := \exp ( - \int_{s_0}^s H(r) \, dr)$, so the equation in question takes the self-adjoint form
\begin{equation}\label{eqn:general-phi-ODE}
    (q \phi')' + \lambda q \phi = 0.
\end{equation}
The endpoint assumptions on $H$ yield $\int_0^{s_0} \frac{ds}{q(s)} = \infty$ and $\int_{s_0}^L \frac{ds}{q(s)} = \infty$. 
In particular $H(s) \to - \infty$ as $s \downarrow 0$ and $H(s) \to + \infty$ as $s \uparrow L$, because $H$ is strictly increasing.
Thus, $q$ is increasing near $0$ and decreasing near $L$, with $\lim_{s \to 0} q(s) = \lim_{s \to L} q(s) = 0$ because the integrals are infinite.

\smallskip \noindent \textbf{Step 1: Properties of a particular endpoint solution.}
We first consider the solution $\psi$ of equation~\eqref{eqn:general-phi-ODE} which is regular at $L$, given by a solving a Volterra equation with boundary behavior
\begin{equation}\label{eqn:volterra-operator}
\lim_{s \uparrow L} \psi(s) = 1, \qquad \lim_{s \uparrow L} q(s) \psi'(s) = 0, \qquad \psi(s) = 1 - \lambda \int_s^L q(r)^{-1} \int_r^L q(t) \psi(t) \, dt \, dr
\end{equation}
near $s$.
Since $m$ is decreasing near $L$, we have $\frac{1}{q(r)} \int_r^L q(t) \, dt \leq L-r$ by the monotonicity of $m$.
Hence, the associated Volterra operator of~\eqref{eqn:volterra-operator} is bounded by $O(\ve)$ in $C( [L-\ve,L])$, so it is a contraction for small $\ve>0$.
We therefore obtain a unique solution on $[L-\ve,L]$, which extends uniquely throughout $(0,L)$.
Moreover, $q(s) \psi'(s) = \lambda \int_s^L q(t) \psi(t) \, dt$ implies that $\psi,\psi'>0$ sufficiently close to $L$, and $\psi'(s) \to 0$ as $s \to L$ by the preceding bound.

If $\psi$ did not reach zero, it would be positive everywhere.
Thus, for any fixed $s_* \in (0,L)$, 
\[ \psi'(s) \geq A q(s)^{-1} \text{ for } 0<s<s_*, \qquad \text{ where } A:= \lambda \int_{s_*}^L q(t) \psi(t) dt > 0,
\]
contradicting $\psi>0$ due to $\int_0^{s_*} \frac{ds}{q(s)} = \infty$.
Thus, $\psi$ has a last zero $\bar{\sigma} \in (0,L)$ and $\psi, \psi'>0$ on $( \bar{\sigma}, L)$.

\smallskip \noindent \textbf{Step 2: Condition to reach a second zero and monotonicity of the slope.}
Next, we claim that any solution $\phi_{\sigma}$ in the statement reaches a second zero $\tau(\sigma)$ if and only if $\sigma < \bar{\sigma}$.
Indeed, equation~\eqref{eqn:general-phi-ODE} implies that the weighted Wronskian of $\phi_{\sigma}, \psi$ is equal to a constant $W$, with $(\frac{\phi_{\sigma}}{\psi})' = \frac{q(\sigma)\psi(\sigma)}{q\psi^2} > 0$ on $(\sigma, L)$, for $\sigma> \bar{\sigma}$.
Hence, no second zero exists.
If $\sigma=\bar\sigma$, then ODE uniqueness makes $\phi_{\sigma} = \frac{\psi}{\psi'(\bar{\sigma})}$, which also remains positive on $(\bar{\sigma},L)$.

Finally, let $\sigma < \bar{\sigma}$.
If $\phi_\sigma$ has a second zero in $(\sigma, \bar\sigma]$, there is nothing to prove. 
Otherwise, $\phi_\sigma(\bar\sigma) > 0$, and evaluating the Wronskian at $\bar\sigma$ gives $W = -q( \bar{\sigma}) \phi_{\sigma}(\bar{\sigma}) \psi'(\bar{\sigma}) <0$. 
Thus, $( \frac{\phi_{\sigma}}{\psi})' = \frac{W}{m \psi^2} < 0$ on $( \bar{\sigma}, L)$.
Moreover, since $\psi$ is bounded near $L$ and $\int_{s_0}^L \frac{ds}{q(s)} = \infty$, we find $\frac{\phi_\sigma}{\psi} \to - \infty$ as $s \uparrow L$, so $\phi_\sigma$ has a zero before $L$. 
Thus, the first zero after $\sigma$, which we denote by $\tau(\sigma)$, exists exactly when $\sigma< \bar\sigma$.

Moreover, applying Lemma~\ref{lemma:uniqueness-solution} to the solutions $\phi_\sigma$, with $\sigma \in (0,\bar{\sigma})$, we know that the terminal slope function $\sigma \mapsto - \phi'_{\sigma}( \tau(\sigma))$ is strictly increasing in $\sigma \in (0, \bar{\sigma})$.

\smallskip \noindent \textbf{Step 3: Asymptotics of the terminal slope.}
    Finally, we prove that 
    \[ \lim_{\sigma \downarrow 0} [ - \phi'_{\sigma} (\tau(\sigma)) ] = 0, \qquad \lim_{\sigma \uparrow \bar{\sigma}} [ - \phi'_{\sigma}(\tau(\sigma)) ] = \infty. \]
    Construct the regular solution $\chi$ at the left endpoint by the
    analogous Volterra equation
    \[
    \chi(s)=1-\lambda\int_0^s\frac{1}{q(r)}
    \int_0^r q(t)\chi(t)\,dt\,dr.
    \]
    It satisfies $\chi(0)=1$ and is positive near $0$.
    Choose $c>0$ small enough so that $2c<\bar\sigma$ and $\chi>0$ on $[0,2c]$.
    For $0<\sigma<c$, the Wronskian identity $( \frac{\phi_{\sigma}}{\chi})' = \frac{\tilde{W}}{q \chi^2}>0$ implies $\phi_\sigma>0$ on $(\sigma,2c]$, hence $2c<\tau(\sigma)<\tau(c)<L$.
    Integrating~\eqref{eqn:general-phi-ODE} in $\{ \phi > 0 \}$ gives
    \[
    0<\phi_\sigma(c)\leq q(\sigma)\int_\sigma^c q(r)^{-1} \, dr.
    \]
    This upper bound tends to zero: for a sufficiently small fixed $\delta<c$,
    the monotonicity of $m$ yields
    \[
    q(\sigma) \int_\sigma^\delta q(r)^{-1} \, dr \leq \delta - \sigma, \qquad q(\sigma) \int_{\delta}^c q(r)^{-1} \, dr \to 0,
    \]
    and then pass $\delta\downarrow 0$. 
    Thus, taking $\limsup_{\sigma \downarrow 0}$ for fixed $\delta$, and then sending $\delta \downarrow 0$, we deduce that $\varphi_{\sigma}(c) \to 0$ as $\sigma \downarrow 0$. 
    On the other hand, evaluating the Wronskian with $\chi$ at $c$ and $\sigma$ gives
    \[
    \phi_\sigma'(c)= \tfrac{\chi'(c)}{\chi(c)} \phi_\sigma(c) + \tfrac{q(\sigma)\chi(\sigma)}{q(c)\chi(c)} \longrightarrow 0.
    \]
    The continuous dependence of the ODE solution on Cauchy data over the fixed compact interval $[c,\tau(c)]$ implies that $\|\phi_\sigma\|_{C^1([c,\tau(c)])}\to0$.
    In particular $|\phi'_{\sigma} (\tau(\sigma))| \to 0$ as $\sigma \downarrow 0$.

    For the second claim, it is clear that $\tau(\sigma) \to L$ as $\sigma \uparrow \bar{\sigma}$: since the map $\sigma \mapsto \tau(\sigma)$ is increasing, the limit $\tilde{L} := \lim_{\sigma \uparrow \bar{\sigma}} \tau(\sigma)$ exists.
    Moreover, $\phi_{\bar{\sigma}}$ is a positive multiple of $\psi$ on $( \bar{\sigma}, L)$, and $\psi>0$ there.
    If $\tilde{L} < L$, then continuous dependence on compact subsets of $(0,L)$ would imply that $\phi_{\bar{\sigma}}(\tilde{L})= 0$, contradicting $\psi(\tilde{L}) > 0$.
    Thus, $\tau(\sigma) \to L$.
    Next, we reverse the independent variable to write $\hat{H}(r) := - H(L-r)$ for $r \in (0,L)$, so $\hat{H}'(r) = H'(L-r) > 0$.
    The two endpoint conditions~\eqref{eqn:blowup-condition} are preserved under this change of variables.
    We also let $\hat{\tau}(\sigma) := L - \tau(\sigma)$ and define a new function $v_{\sigma}(r) := \frac{\phi_{\sigma}(L-r)}{- \phi'_{\sigma}(\tau(\sigma))}$, which satisfies the modified equation
    \begin{equation}\label{eqn:v-new-ODE}
    v''_{\sigma} - \hat{H} v'_{\sigma} + \lambda v_{\sigma} = 0 \qquad v_{\sigma}(\hat{\tau}(\sigma)) = 0, \quad v'_{\sigma}(\hat{\tau}(\sigma))= 1.
    \end{equation}
    By construction $v_{\sigma}$ has a second zero at $L - \sigma$, where $v'_{\sigma} (L-\sigma) = \frac{1}{\phi'_{\sigma}(\tau(\sigma))} < 0$.
    Since $\hat{\tau}(\sigma) = L - \tau(\sigma) \to 0$, we may use the result of the first claim for the equation~\eqref{eqn:v-new-ODE} with coefficient $\hat{H}$.
    It follows that $v'_{\sigma} (L - \sigma) \to 0$ as $\sigma \uparrow \bar{\sigma}$, hence $\frac{1}{\phi'_{\sigma}(\tau(\sigma))} \to 0$ and $\phi'_{\sigma}(\tau(\sigma)) \to - \infty$ as desired.

    Combining the above discussion, we deduce that the function $\sigma \mapsto - \phi'_{\sigma}(\tau(\sigma))$ defines a continuous, strictly increasing map $(0, \bar{\sigma}) \to ( 0,\infty)$ with $\lim_{\sigma \downarrow 0} [ - \phi'_{\sigma}(\tau(\sigma)) ] = 0$ and $\lim_{\sigma \uparrow \bar{\sigma}} [ - \phi'_{\sigma}(\tau(\sigma)) ] = \infty$. 
    Thus, this map is a bijection, so there exists a unique $\sigma_* \in (0, \bar{\sigma})$ with $\phi'_{\sigma_*} ( \tau(\sigma_*)) = - 1$.
\end{proof}

\begin{corollary}
    \label{cor:existence-and-uniqueness}
For every isoparametric foliation $\{ M_s \}$ of $\bS^{n-1}$, there exists a unique solution $\phi$ to \eqref{eq:extremal} that is constant along leaves of the foliation such that $\{\phi>0\}$ has disconnected boundary.
\end{corollary}

\begin{proof}
By \ref{prop:general-uniqueness-argument}, it suffices to verify the endpoint condition~\eqref{eqn:blowup-condition}: since $v(s) = C( \sin \frac{gs}{2})^{m_1} ( \cos \frac{gs}{2})^{m_2}$ has $v(s) \asymp s^{m_1}$ near $s=0$ and $v(s) \asymp ( \frac{\pi}{g} - s)^{m_2}$ near $s = \frac{\pi}{g}$, we use $H = - \frac{v'}{v}$ to obtain
\[
\exp \Bigl( \int_{s_0}^s H(r) \, dr \Bigr) = \frac{C'}{v(s)}, \qquad \int_0 \frac{ds}{v(s)} = \int^{\frac{\pi}{g}} \frac{ds}{v(s)} = \infty
\]
due to $m_i \geq 1$.
Thus, there exists a unique profile $\phi_{\sigma_*}$ with $\phi'_{\sigma_*}(\sigma_*) = -\phi'_{\sigma_*}(\tau(\sigma_*)) = 1$.
\end{proof}

We finally combine these properties to complete the proof of Theorem~\ref{thm:unified-theorem}, showing that the geometry of the radial graph determines the homogeneous solution $U$.
\begin{proof}[Proof of Theorem~\ref{thm:unified-theorem}]
We define the radial projection map $\pi: \Sigma_U \ni x \mapsto \frac{x}{|x|} \to \bS^{n-1}$ on the radial graph.
Since $U$ is positive and one-homogeneous, $\pi$ is a diffeomorphism from $\Sigma_U$ onto the spherical positivity set $\Omega^S = \{ U > 0 \} \cap \bS^{n-1}$, with inverse $\pi^{-1}(\omega) = \frac{\omega}{U(\omega)}$.
By our assumption, $\Omega^S$ is foliated by the leaves of $M_s$ and the radius of the unique point of $\Sigma_U$ over $\omega$ is constant along every leaf.
Thus, $U(\omega)$ is constant along leaves of the foliation, so $U(\rho \omega) = \rho \phi (s(\omega))$ for some $\phi$.

By Corollary~\ref{cor:existence-and-uniqueness} and the discussions above, the solution is either one of the Type I profiles $U_{M_1}, U_{M_2}$, or it is the unique Type II solution $\bar{U}_M$.
For the Type I profiles, the maximum is achieved at the relevant focal manifold, and $\Sigma_U$ has one end.
For a Type II solution, the profile is strictly log-concave, due to Proposition~\ref{prop:general-uniqueness-argument}.
Therefore, it attains a unique interior maximum, along a regular leaf, and $\Sigma_U$ has two ends.
This establishes our geometric characterization.

The application to isoparametric foliations generated by cohomogeneity-one actions of $G \subset O(n)$ on $\bS^{n-1}$ is immediate: if $U(gx) = U(x)$ for all $g \in G$, then $G$-invariance makes the spherical profile descend to the orbit interval.
In particular, suppose $\textup{SO}(n-1) \subseteq \textup{Isom}(\Sigma_U)$. 
Then by Proposition \ref{prop:intrinsic-ambient-isometry}, we have $U(gx) = U(x)$ for all $g\in \textup{SO}(n-1)$. The only $\textup{SO}(n-1)$-invariant solutions are the flat half-space solution and the axisymmetric De Silva-Jerison cone.
This completes the proof.
\end{proof}

\section{Densities of homogeneous solutions and asymptotic analysis}\label{section:density-of-cones}

We now analyze the density of the isoparametric cones $U_{M_i}, \bar{U}_M$ constructed in Theorem~\ref{thm:isoparametric-foliation-one-phase-cone}.
Given the symmetry between the profiles of the cones $U_{M_i}$ and density computations, we study the cones $U_{M_1}$ and $\bar{U}_M$ and denote their densities by
\[
\Theta_{g,m_1,m_2} := \Theta ( U_{M_1}), \qquad \bar{\Theta}_{g,m_1,m_2} := \Theta (\bar{U}_M)
\]
for $(g,m_1,m_2)$ the parameters of the isoparametric foliation.
In particular, the axisymmetric De Silva-Jerison cone on $\bR^n$ is of Type II, in the sense of Theorem~\ref{thm:isoparametric-foliation-one-phase-cone}, with density $\bar{\Theta}_{1,n-2,n-2}$.

We first perform a change of variables that will be useful in the ensuing computations.
\begin{lemma}\label{lemma:density-cone-computation}
Let us consider an $\cF$-invariant solution of the shape optimization problem~\eqref{eq:homsph} with profile function $f(t)$ having positivity interval $I$.
    The corresponding homogeneous solution of the one-phase problem, given by
    \[
    U(x) := c \, |x| f \bigl( \sin \tfrac{gs(\omega)}{2} \bigr), \quad \text{when } g \geq 2, \quad \text{or} \quad U(x) := c \, |x| f (\cos s(\omega)), \quad \text{when } g=1, 
    \]
    has density $\Theta(U)$ given by
    \begin{equation}\label{eqn:density-of-cone}
    \begin{split}
        \Theta(U) &= \frac{\int_I t^{m_1} (1-t^2)^{\frac{m_2-1}{2}} \, dt}{\int_0^1 t^{m_1}(1-t^2)^{\frac{m_2-1}{2}} \, dt} \quad \text{if } \; g \geq 2, \qquad \Theta(U) = \frac{\int_I (1-t^2)^{\frac{n-3}{2}} \, dt}{\int_{-1}^1 (1-t^2)^{\frac{n-3}{2}} \, dt} \quad \text{if } \; g=1.
    \end{split}
    \end{equation}
\end{lemma}
\begin{proof}
    For $g \geq 2$, the map $s \mapsto t := \sin \frac{gs}{2}$ is strictly increasing in $s \in [0,\frac{\pi}{g}]$, so the positive phase of the cone $U$ on the sphere is precisely $\{ \omega \in \bS^{n-1} : t(\omega) \in I \}$.
    The spherical density along the foliation is given by~\eqref{eqn:isoparametric-volume-density}, namely 
    \[
    d \cH^{n-1} = C (\sin \tfrac{gs}{2})^{m_1} ( \cos \tfrac{gs}{2})^{m_2} \, ds = \tfrac{2C}{g} t^{m_1}(1-t^2)^{\frac{m_2-1}{2}} \, dt.
    \]
    The expression~\eqref{eqn:density-of-cone} follows from integrating this measure along the positivity interval $I$.
    For $g=1$, the map $s \mapsto t := \cos s$ is strictly decreasing in $s \in [0,\pi]$ and the induced measure is $d \cH^{n-1} = C (\sin s)^{n-2} \, ds = C (1-t^2)^{\frac{n-3}{2}} \, dt$ with $t = \cos s$, leading to the second expression.
\end{proof}

\begin{proposition}\label{proposition:axisymmetric-density}
    The density of the axisymmetric De Silva-Jerison cone satisfies the asymptotics
    \begin{equation}\label{eq:density-DSJ-asymp}
         \lim_{n \to \infty} \bar{\Theta}_{1,n-2,n-2} = \frac{\int_0^c e^{- \frac{x^2}{2}} \, dx}{\int_0^{\infty} e^{- \frac{x^2}{2}} \, dx} = \textup{erf} \Bigl( \frac{c}{\sqrt{2}} \Bigr) \simeq 0.80876, \qquad c \simeq 1.3069,
    \end{equation}
    where $c$ is the unique positive zero of the function ${}_1F_1( - \frac{1}{2}; \frac{1}{2} ; \frac{x^2}{2})$ and $\textup{erf}(t) := \frac{2}{\sqrt{\pi}} \int_0^t e^{-x^2} \, dx$.

    More generally, for $m_1$ kept fixed and $m_2 \to \infty$, we have
    \[
    \lim_{m_2 \to \infty} \Theta_{2,m_1,m_2} = \frac{1}{\Gamma( \frac{m_1+1}{2})} \int_0^{\rho_{m_1}} y^{\frac{m_1-1}{2}} e^{-y} \, dy, \qquad \text{where } \; {}_1 F_1 \Bigl( - \frac{1}{2}, \frac{m_1+1}{2} ; \rho_{m_1} \Bigr) = 0.
    \]
    Finally, for $m_2$ kept fixed and $m_1 \to \infty$, the limiting density is
    \[
    \lim_{m_1 \to \infty} \Theta_{2,m_1,m_2} = \frac{1}{\Gamma( \frac{m_2+1}{2})} \int_{\sigma_{m_2}}^{\infty} y^{\frac{m_2-1}{2}} e^{-y} \, dy, \qquad \text{where } \; U \Bigl( - \frac{1}{2}, \frac{m_2+1}{2} ; \sigma_{m_2} \Bigr) = 0.
    \]
    Here, ${}_1F_1(a,b;y)$ and $U(a,b;y)$ denote, respectively, the regular Kummer solution and the Tricomi solution of the confluent hypergeometric equation $y F'' + (b-y) F' - a F= 0$, with asymptotic behaviors ${}_1F_1(a,b;y) = 1 + O(y)$ as $y \to 0$ and $U(a,b;y) \sim y^{-a}$ as $y \to \infty$.
\end{proposition}
\begin{proof}
    We follow the computations in~\cite{FTW-stability-one-phase}*{\S 4.1}.
    The axisymmetric cone has profile function defined on the interval $[-t_n, t_n]$, where $t_n$ is the zero of the function $F_n(t) := {}_2F_1( \frac{n-1}{2} , - \frac{1}{2} ; \frac{1}{2} ; t^2)$.
    Let $G_n(x) := F_n( \frac{x}{\sqrt{n-1}})$, so the hypergeometric series give the convergence $G_n \to G_{\infty}$ in $C^{\infty}_{\text{loc}}$, where the limiting function $G_{\infty} = {}_1 F_1 ( - \frac{1}{2} ; \frac{1}{2} ; \frac{x^2}{2})$ is the solution of $G''_{\infty} - x G'_{\infty} + G_{\infty} = 0$ with $(G_{\infty}(0), G'_{\infty}(0)) = (1,0)$.
    We find that $G''_{\infty} (x) = - \exp ( \frac{x^2}{2})$, so $G_{\infty}$ is concave and has a unique zero $c$, which satisfies ${}_1F_1 ( -\frac{1}{2} ; \frac{1}{2} ; \frac{c^2}{2}) =0$.
    Thus, $\sqrt{n-1} \, t_n \to c$, and $x = \sqrt{n-1} \, t$ produces the $C^{\infty}_{\text{loc}}$ limit $(1-t^2)^{\frac{n-3}{2}} = \bigl( 1 - \tfrac{x^2}{n-1} \bigr)^{\frac{n-3}{2}} \xrightarrow{C^{\infty}_{\textup{loc}}} e^{- \frac{x^2}{2}}$.
    Applying this relation to the integral~\eqref{eqn:density-of-cone} and sending $n \to \infty$ proves the claimed equality.

    To study the limits as $m_1 \to \infty$ or $m_2 \to \infty$, write $F_{m_1, m_2}(t) := {}_2F_1 ( \frac{m_1+m_2+1}{2} , - \frac{1}{2} ; \frac{m_1+1}{2} ; t^2)$ and denote the first zero of $F_{m_1,m_2}$ as $t_{m_1,m_2}$.
    For $m_1$ fixed, we set $H_{m_2}(y) := F_{m_1, m_2} ( \sqrt{\frac{2y}{m_2}})$, so the hypergeometric coefficients in the expansion of $H_{m_2}(y)$ satisfy $( \frac{m_1+m_2+1}{2})_k ( \frac{2}{m_2})^k \to 1$.
    This implies
    \[
    H_{m_2}(y) = \sum_{k=0}^{\infty} \frac{(-\frac{1}{2})_k}{( \frac{m_1+1}{2})_k k!} \Bigl( \frac{m_1+m_2+1}{2} \Bigr)_k \Bigl( \frac{2y}{m_2} \Bigr)^k \to H_{\infty}(y) =: {}_1F_1 \Bigl( - \frac{1}{2} ; \frac{m_1+1}{2}; y \Bigr)
    \]
    in $C^{\infty}_{\textup{loc}}( [0,\infty))$.
    Let $\rho_{m_1}$ be the first positive zero of $H_{\infty}$, which is simple, so $\frac{m_2}{2} t^2_{m_1, m_2} \to \rho_{m_1}$ due to the above locally uniform convergence.
    With the change of variables $y=\frac{m_2}{2}t^2$, we obtain $t^{m_1}(1-t^2)^{\frac{m_2-1}{2}} \, dt = y^{\frac{m_1-1}{2}} (1- \frac{2y}{m_2})^{\frac{m_2-1}{2}} dy$ up to a constant factor that cancels between the numerator and the denominator.
    Moreover, $(1 - \frac{2y}{m_2})^{\frac{m_2-1}{2}} \to e^{-y}$ on compact sets as $m_2 \to \infty$.
    Thus, applying the dominated convergence theorem to the density formula~\eqref{eqn:density-of-cone}, we obtain 
    \[ 
    \lim_{m_2 \to \infty}\Theta_{2,m_1,m_2} =\lim_{m_2 \to \infty}\frac{ \displaystyle{ \int_0^{\frac{m_2}{2} t^2_{m_1,m_2}}} y^{\frac{m_1-1}{2}} \bigl(1-\tfrac{2y}{m_2} \bigr)^{\frac{m_2-1}{2}}\,dy}{ \displaystyle{\int_0^{\frac{m_2}{2}}}y^{\frac{m_1-1}{2}}\bigl(1-\tfrac{2y}{m_2}\bigr)^{\frac{m_2-1}{2}}\,dy} = \frac{1}{\Gamma(\frac{m_1+1}{2})}\int_0^{\rho_{m_1}}y^{\frac{m_1-1}{2}}e^{-y}\,dy
    \]
    as claimed.
    For the second limit, we set $y = \frac{m_1+1}{2}(1-t^2)$ and define $K_{m_1}(y) := F_{m_1, m_2} ( \sqrt{1 - \frac{2y}{m_1+1}})$.
    By the standard confluence of the Gauss hypergeometric function to the Tricomi function (see~\cite{dtmf} 15.10.21, 16.8.10, 5.11.12, and 13.2.42), we have 
    \[
    \textstyle{\sqrt{\tfrac{m_1+1}{2}}} \, K_{m_1} \to U \bigl(- \tfrac{1}{2}; \tfrac{m_2+1}{2} ; y \bigr) \qquad \text{in } \quad C^{\infty}_{\textup{loc}}(0,\infty), \qquad \text{as } \; m_1 \to \infty.
    \]
    All the non-constant coefficients defining $F_{m_1,m_2}$ are negative, so this function is strictly decreasing on $(0,1)$, and $K_{m_1}$ is strictly increasing.
    Moreover, $ \frac{d}{dy} U ( - \frac{1}{2}; \frac{m_2+1}{2} ;y) = \frac{1}{2} U(\frac{1}{2}; \frac{m_2+1}{2}+1;y)>0$ and
    \[
   U( - \tfrac{1}{2}; \tfrac{m_2+1}{2} ; y) \to - \infty \quad \text{as } \; y \downarrow 0, \qquad U ( - \tfrac{1}{2} ; \tfrac{m_2+1}{2} ; y) \sim y^{\frac{1}{2}} \quad \text{as } \; y \to \infty,
    \]
    so there is a unique positive zero $\sigma_{m_2}$,which is simple.
    Therefore,
    \[
    \tfrac{m_1+1}{2} (1 - t^2_{m_1,m_2}) \to \sigma_{m_2}, \qquad t^{m_1} (1-t^2)^{\frac{m_2-1}{2}} \, dt = y^{\frac{m_2-1}{2}} (1 - \tfrac{2y}{m_1+1})^{\frac{m_1-1}{2}} \, dy,
    \]
    up to a constant factor that cancels between the numerator and denominator.
    Since $(1 - \frac{2y}{m_1+1})^{\frac{m_1-1}{2}} \to e^{-y}$ locally uniformly, we may again apply dominated convergence to obtain
    \[   
    \lim_{m_1 \to \infty} \Theta_{2,m_1,m_2}= \lim_{m_1 \to \infty} \frac{ \displaystyle{\int_{\frac{m_1+1}{2} (1 - t^2_{m_1,m_2})}^{\frac{m_1+1}{2}}}y^{\frac{m_2-1}{2}} \bigl(1-\tfrac{2y}{m_1+1} \bigr)^{\frac{m_1-1}{2}}\,dy}{ \displaystyle{\int_0^{\frac{m_1+1}{2}}} y^{\frac{m_2-1}{2}} \bigl(1-\tfrac{2y}{m_1+1} \bigr)^{\frac{m_1-1}{2}}\,dy} =\frac{1}{\Gamma(\frac{m_2+1}{2})}\int_{\sigma_{m_2}}^\infty y^{\frac{m_2-1}{2}}e^{-y}\,dy.
    \]
    This completes the proof.
\end{proof}

On the other hand, when both $m_1, m_2 \to \infty$, we have the following explicit limit.
\begin{proposition}\label{prop:m1,m2-to-infty}
    As $m_1, m_2 \to \infty$, the limiting density satisfies
    \[
    \lim_{m_1, m_2 \to \infty} \Theta_{2,m_1,m_2} = \Phi(z_*) \simeq 0.7779, \qquad z_* \simeq 0.7650.
    \]
    Here, $z_* > 0$ is the first zero of the parabolic cylinder function $D_{1/2}(-z)$, and $\Phi(z)$ is the standard Gaussian cumulative distribution function.
\end{proposition}
\begin{proof}
We again work in the variable $t = \sin \frac{g s(\omega)}{2}$, for $g \geq 2$, or $t = \cos s(\omega)$, for $g=1$.
Writing $q(t) := t^{m_1} (1-t^2)^{\frac{m_2-1}{2}}$, we can transform the equation~\eqref{eqn:laplacian-of-f} into a self-adjoint expression satisfied by the hypergeometric function, given by
    \[
    \bigl( (1-t^2) q(t) F'(t) \bigr)' + (n-1) q(t) F(t) = 0, \qquad F(t) := {}_2 F_1 \bigl( \tfrac{n-1}{2} , - \tfrac{1}{2} ; \tfrac{m_1+1}{2} ; t^2 \bigr).
    \]
    Since $\frac{d}{dt} \log q(t) = \frac{m_1}{t} - \frac{(m_2 - 1) t}{1-t^2}$, the function $q(t)$ attains its maximum at $t_0 = \sqrt{\frac{m_1}{m_1+m_2-1}}$. For simplicity of notation let $N:=m_1+m_2-1$. Simple computation shows 
    \begin{equation}\label{tmp:der-log-weight}
        (\log q)'(t_0)=0, \qquad (\log q)''(t_0) = - \tfrac{2N^2}{m_2-1}.
    \end{equation} 
    We will recenter $q(t)$ around the maximum $t_0$ and rescale to show that this procedure approximates the standard Gaussian. 
    Let us introduce a variable $x$ by
    \begin{equation}\label{eqn:t-t0-transformation}
    t := t_0 + \tfrac{1}{N} \textstyle {\sqrt{\tfrac{m_2-1}{2}}} \, x, \qquad \tilde{F}(x) := F(t), \qquad \tilde{q}(x) := \frac{q(t)}{q(t_0)}.
\end{equation}
    Then, the left and right endpoints $t = 0,1$, correspond respectively to $x_L = - \sqrt{\tfrac{2}{m_2 -1}} N t_0$ and $x_R = \sqrt{\tfrac{2}{m_2-1}} N (1 -t_0)$, which produces
    \[
    x_L^2  = 2N \tfrac{m_1}{m_2-1} \geq 2m_1, \qquad x_R^2 = \tfrac{2(m_2-1)}{(1 + t_0)^2} \geq \tfrac{m_2-1}{2}.
    \]
    In particular, this implies that $x_L \to - \infty$ and $x_R \to + \infty$ when both $m_1,m_2 \to + \infty$.
    In terms of the variable $x$, the equation for $F$ transforms into
    \begin{equation}\label{eq:ode-recentered}
    \tfrac{1}{1- t_0^2} \bigl( (1-t(x)^2) \tilde{q}(x) \tilde{F}'(x) \bigr)' + \tfrac{N+2}{2N} \tilde{q}(x) \tilde{F}(x) = 0.
    \end{equation}
    Here we use the identity $m_1+m_2 = n-2$ when $g=2$.
   
    We use the property \eqref{tmp:der-log-weight} together with a Taylor expansion of $\log q(t)$ around $t_0$, and then apply the change of variables \eqref{eqn:t-t0-transformation} to deduce that $\log \tilde{q}(x) = - \frac{x^2}{2} + O_R(m_1^{- \frac{1}{2}} + m_2^{- \frac{1}{2}})$ for $|x| \leq R$. 
    Because this property also holds for iterated derivatives, we deduce that $\tilde{q}(x) \to \exp ( - \frac{1}{2} x^2)$ in $C^{\infty}_{\textup{loc}}(\bR)$.
    Thus, the limiting ODE of \eqref{eq:ode-recentered} becomes
    \begin{equation}\label{eq:parabolic-cylinder}
        ( e^{- x^2/2} \Psi')' + \tfrac{1}{2} e^{- x^2/2} \Psi = 0, \qquad \text{ or equivalently, } \qquad \Psi'' - x \Psi' + \tfrac{1}{2} \Psi = 0
    \end{equation}
    on compact intervals.
    Moreover, the property $F(0) = 1$ together with its monotonicity and concavity determines the behavior of the limit function as $x_L \to - \infty$: after renormalization, we must have
    \[ 
    \frac{\tilde{F}(x)}{\tilde{F}(0)} \to \Psi(x) := \frac{e^{x^2/4} D_{1/2}(-x)}{D_{1/2}(0)}, 
    \]
    where $D_{1/2}(z)$ denotes the corresponding parabolic cylinder solution to \eqref{eq:parabolic-cylinder}.
    See also the standard form in~\cite{dtmf}*{12.2.4}. 
    Therefore, this $\Psi$ is the solution of equation~\eqref{eq:parabolic-cylinder} with polynomial growth at $-\infty$, given by $\Psi(x) \sim |x|^{\frac{1}{2}}$.
    Letting $z_* \simeq 0.7650$ denote the first zero of $\Psi$ or equivalently $D_{1/2}(-x)$, we deduce from~\eqref{eqn:t-t0-transformation} that the zero $x_*$ of $\tilde{F}(x)$ satisfies $x_* \to z_*$.
    Using the properties $x_L \to -\infty, x_R \to + \infty$, $\tilde{q}(x) \to \exp( - \frac{1}{2} x^2)$, and a change of variables, we conclude that
    \[
    \Theta_{2,m_1,m_2} \to \frac{\int_{- \infty}^{z_*} e^{-x^2/2} \, dx}{\int_{-\infty}^{\infty} e^{-x^2/2} \, dx} = \Phi(z_*)
    \]
    where $\Phi(z)$ is the standard Gaussian cumulative distribution function.
    The value $z_* \simeq 0.7650$ gives $\Phi(z_*) \simeq 0.7779$.
    This completes the proof of our assertion.
\end{proof}

Propositions \ref{proposition:axisymmetric-density} and \ref{prop:m1,m2-to-infty} show that, for all sufficiently large $m,n$,
\[
\Theta_{2 , m, n } < \Theta_{2,1 , m+n - 1 } < \bar{\Theta}_{1, m+n,m+n}.
\]
In particular, for even $n$, we obtain
\[
\Theta_{2, \frac{n}{2} - 1, \frac{n}{2} - 1} < \Theta_{2,1,n-3} < \bar{\Theta}_{1, n-2, n-2}.
\]
The middle density is computed by taking $m_1=1$ in Proposition~\ref{proposition:axisymmetric-density}, where the positive zero $\rho_1$ of ${}_1F_1( - \frac{1}{2}; 1;\rho_1) = 0$ is numerically $\rho_1 \simeq 1.57996$.
Using $\Gamma(1)=1$, the limiting density is therefore 
\[
\lim_{n\to\infty}\Theta_{2,1,n-3} = \lim_{m_2\to\infty}\Theta_{2,1,m_2} = \int_0^{\rho_1}e^{-y}\,dy =1-e^{-\rho_1} \simeq 0.7940.
\]
In fact, the estimates of~\cite{FTW-stability-one-phase}*{\S 4.1} can be used to quantify the inequality $\Theta_{2,1,n-3} < \bar{\Theta}_{1,n-2,n-2}$ with asymptotic bounds in $n$.
Combined with a finite check for all $n \geq 5$, this shows that $\Theta_{2,1,n-3} < \bar{\Theta}_{1,n-2,n-2}$ in dimension $5$ and above.

\begin{table}[htbp]
\centering
\begin{tabular}{c|c|ccccc}
$n$ & $4$ & $5$ & $10$ &$15$ & $20$ & $25$\\
\hline
$\Theta_{2,m_1, m_2}$ & $0.8261$ & $0.8075$ & $0.7917$ & $0.7863$ & $0.7841$ & $0.7827$ \\
\hline
$\Theta_{2,1,n-3}$ & $0.8261$ & $0.8075$ & $0.7959$ & $0.7947$ & $0.7944$ & $0.7942$ \\
\hline
$\bar{\Theta}_{1,n-2,n-2}$ & $0.8183$ & $0.8138$ & $0.8096$ & $0.8091$ & $0.8089$ & $0.8089$ \\
\hline
\end{tabular}
\caption{Numerical comparison of the densities, rounded to four decimal places. For even $n$, the first row uses $m_1=m_2=\frac{n}{2}-1$; for odd $n$, we use $(m_1, m_2) = (\frac{n-3}{2} , \frac{n-1}{2})$.}
\label{table:density}
\end{table}

Still, we conjecture that the De Silva-Jerison cone does minimize the density in low dimensions.
\begin{conjecture}\label{conjecture:dims-3-and-4}
The axisymmetric cone minimizes the density among non-flat cones in $\bR^3$ and $\bR^4$.
\end{conjecture}
The $3$-dimensional result would represent significant progress, establishing an analogue of the Willmore conjecture for the one-phase problem~\cite{willmore-conjecture}.
Parise and Zhu~\cite{parise-zhu} recently proved the uniqueness of the De Silva-Jerison cone in $\mathbb{R}^3$ among solutions with annular spherical trace.
In dimension $3$, the computations~\eqref{eqn:hypergeometric-isoparametric} and~\eqref{eqn:density-of-cone} yield the density of the axisymmetric cone as
\[
f_{ \bS^1 \subset \bS^2 }(t) = 1 - t \tanh^{-1}(t), \qquad \bar\Theta_{1,1,1} = t_{\bS^1 \subset \bS^2} \simeq 0.834 \, .
\]
In dimension $4$, we compare the densities of the axisymmetric De Silva-Jerison cone and the two $O(2) \times O(2)$-invariant solutions (up to ambient isometries) produced by Theorem~\ref{thm:isoparametric-foliation-one-phase-cone}, corresponding to the foliation by Clifford tori $\bS^1 \times \bS^1 \subset \bS^3$.
Using the explicit expressions of the solutions $U_{\bS^1 \subset \bS^3}$ and $\bar{U}_{\bS^1 \times \bS^1 \subset \bS^3}$ from~\eqref{eqn:Phi-M1}, we compute that
\[
\bar\Theta_{1,2,2} \simeq 0.818, \qquad \Theta_{2,1,1} \simeq 0.826, \qquad \bar{\Theta}_{2,1,1} \simeq 0.975,
\]
which supports the conjecture among these explicit examples.
The only other known homogeneous solutions are the ones constructed by Hines-Kolesar-McGrath in $\bS^2$ and $\bS^3$ in~\cite{hines-kolesar-mcgrath}.
Roughly speaking, these solutions are supported on spherical domains $\Omega^{\textup{HKM}}_m$, constructed as a small perturbation of $\bS^2$ or $\bS^3$ with a large number of small geodesic disks removed along the equator.
It follows easily from their construction that the corresponding densities are sufficiently close to 1.

\begin{lemma}\label{lemma:density-HKM}
    The Hines-Kolesar-McGrath solutions in $\mathbb{R}^3$ have density 
    \[
    \Theta^{\textup{HKM}, 3}_m = 1 - O( m^{- \frac{1}{2}} e^{- \sqrt{2m}} ) \qquad \text{as } \; m \to \infty.
    \]
    The corresponding solutions in $\bR^4$ have density
    \[
\Theta^{\textup{HKM},4}_m=1-\tfrac{2\pi^2}{3m^4}F^3+O(m^{-5}) = 1-O(m^{-4}), \qquad \text{ as } m\to \infty,
\]
where $F\in(2.18,2.19)$ is the constant appearing in~\cite{hines-kolesar-mcgrath}*{Theorem 1.2}.
\end{lemma}
\begin{proof}
The computation is similar in both cases; for brevity, we will use the same notation in dimensions $3$ and $4$, in the respective step.

\smallskip \noindent \textbf{Dimension $3$.}
To simplify notation, let $D_m := \bigcup_{L_0} D(\tau_0) \cup \bigcup_{L_2} D(\tau_2)$, where $D_{L_i}(\tau_i)$ is the $\tau_i$-neighborhood of the set $L_i$ in $\bS^2$.
   Following~\cite{hines-kolesar-mcgrath}*{Theorem 1.1} for large $m$, we have
   \begin{equation}\label{eqn:tau-zero-tau-2}
       \begin{split}
           L_0 &= \Bigl\{ \bigl( \cos \tfrac{2 \pi j}{m}, \sin \tfrac{2 \pi j}{m} , 0 \bigr) \Bigr\}_{j\in \mathbb{Z}}, \qquad L_2 = \{ (0,0,\pm 1)\}, \\
           \tau_0 &= c_m m^{- \frac{3}{4}} e^{- \sqrt{\frac{m}{2}} }, \qquad \tau_2 = d_m m^{- \frac{1}{4}} e^{- \sqrt{\frac{m}{2}}}, \qquad c_m, d_m = O(1). 
       \end{split}
   \end{equation}
Moreover, we can express the positive phase $\Omega^{\textup{HKM,3}}_m$ of the solution as a normal graph over $\bS^2 \setminus D_m$ with graph function $v$ satisfying $\| v \|_{C^{2,\alpha}(\partial D_m)} \leq C \tau_2^{\frac{5}{2}}$, hence by \eqref{eqn:tau-zero-tau-2}
   \[
   1 - \Theta_m^{\textup{HKM},3}  = \tfrac{1}{4} (m \tau_0^2 + 2 \tau_2^2) + O( m \tau_0^4 + \tau_2^4 + \tau_2^{\frac{5}{2}}) = \tfrac{1}{4} m^{- \frac{1}{2}} e^{- \sqrt{2m}} (c_m^2 + 2d_m^2) + o ( m^{- \frac{1}{2}} e^{- \sqrt{2m}}).
   \]
\noindent \textbf{Dimension $4$.}
   For the construction of~\cite{hines-kolesar-mcgrath}*{Theorems 1.2 and 6.28} in $\bR^4$, the domain $\Omega^{\textup{HKM}}_m$ is now a normal perturbation over $\bS^3 \setminus D_m$, with $D_m := D_{L_m}(\tau_m)$ for
   \[
L_m= \Bigl\{ \tfrac{1}{\sqrt{2}}\bigl(e^{\frac{2\pi i j}{m}},e^{\frac{2\pi i k}{m}}\bigr) \Bigr\}_{j,k\in\bZ} \subset\bS^3, \qquad \tau_m = \bigl( \tfrac{m^2}{\pi F}+c_m m \bigr)^{-1}, \qquad c_m = O(1).
\]
The boundary of $\Omega^{\textup{HKM,4}}_m$ is a normal graph over $\partial D_m$ with graph function $v_m$ which satisfies $\|v_m\|_{C^{2,\alpha}(\partial D_m) }\leq C\tau_m^{\frac{5}{2}}$. 
Distinct points of $L_m$ are separated by a distance comparable to $m^{-1}$, whereas $\tau_m=O(m^{-2})$, so for large $m$, the set $D_m$ is the disjoint union of $m^2$ geodesic balls of radius $\tau_m$.
The volume of a geodesic ball of radius $\tau$ in the unit sphere $\bS^3$ is
\[
\cH^3(B_{\bS^3}(\tau))=4\pi\int_0^\tau\sin^2r\,dr=2\pi(\tau-\sin\tau\cos\tau)= \tfrac{4\pi}{3}\tau^3+O(\tau^5);
\]
and $\cH^2(\partial D_m)=4\pi m^2\sin^2\tau_m=O(m^2\tau_m^2)$.
The estimate on $v_m$ and $\cH^3(\bS^3) = 2\pi^2$ therefore give
\begin{align*}
\left|\cH^3(\bS^3\setminus\Omega^{\textup{HKM},4}_m)-\cH^3(D_m)\right| \leq C\|v_m\|_{C^0(\partial D_m)}\cH^2(\partial D_m)=O(m^2\tau_m^{\frac{9}{2}}), \\
1-\Theta^{\textup{HKM},4}_m= \tfrac{m^2}{\pi}(\tau_m-\sin\tau_m\cos\tau_m)+O(m^2\tau_m^{\frac{9}{2}}) = \tfrac{2}{3\pi}m^2\tau_m^3+O(m^2\tau_m^{\frac{9}{2}}+m^2\tau_m^5).
\end{align*}
Finally, the desired estimate follows by substituting $\tau_m=\frac{\pi F}{m^2}\left(1+\frac{\pi F c_m}{m}\right)^{-1}=\frac{\pi F}{m^2}+O(m^{-3})$.
\end{proof}

\bibliography{ref}

\end{document}